\documentclass[11pt]{article}
\usepackage{latexsym,amsfonts,amssymb,amsmath,amsthm}
\usepackage{graphicx}

\usepackage[usenames,dvipsnames]{color}
\usepackage{ulem}

\usepackage{color}

\begin{document}

\newtheorem{tm}{Theorem}[section]
\newtheorem{prop}[tm]{Proposition}
\newtheorem{defin}[tm]{Definition}
\newtheorem{coro}[tm]{Corollary}
\newtheorem{lem}[tm]{Lemma}
\newtheorem{assumption}[tm]{Assumption}
\newtheorem{rk}[tm]{Remark}
\newtheorem{nota}[tm]{Notation}
\numberwithin{equation}{section}

\newcommand{\stk}[2]{\stackrel{#1}{#2}}
\newcommand{\dwn}[1]{{\scriptstyle #1}\downarrow}
\newcommand{\upa}[1]{{\scriptstyle #1}\uparrow}
\newcommand{\nea}[1]{{\scriptstyle #1}\nearrow}
\newcommand{\sea}[1]{\searrow {\scriptstyle #1}}
\newcommand{\csti}[3]{(#1+1) (#2)^{1/ (#1+1)} (#1)^{- #1
 / (#1+1)} (#3)^{ #1 / (#1 +1)}}
\newcommand{\RR}[1]{\mathbb{#1}}

\newcommand{\rd}{{\mathbb R^d}}
\newcommand{\ep}{\varepsilon}
\newcommand{\rr}{{\mathbb R}}
\newcommand{\alert}[1]{\fbox{#1}}
\newcommand{\eqd}{\sim}
\def\p{\partial}
\def\R{{\mathbb R}}
\def\N{{\mathbb N}}
\def\Q{{\mathbb Q}}
\def\C{{\mathbb C}}
\def\l{{\langle}}
\def\r{\rangle}
\def\t{\tau}
\def\k{\kappa}
\def\a{\alpha}
\def\la{\lambda}
\def\De{\Delta}
\def\de{\delta}
\def\ga{\gamma}
\def\Ga{\Gamma}
\def\ep{\varepsilon}
\def\eps{\varepsilon}
\def\si{\sigma}
\def\Re {{\rm Re}\,}
\def\Im {{\rm Im}\,}
\def\E{{\mathbb E}}
\def\P{{\mathbb P}}
\def\Z{{\mathbb Z}}
\def\D{{\mathbb D}}
\newcommand{\ceil}[1]{\lceil{#1}\rceil}

\title{Persistence and convergence in  parabolic-parabolic chemotaxis system with  logistic source on $\mathbb{R}^{N}$}

\author{
Wenxian Shen and Shuwen Xue \\
Department of Mathematics and Statistics\\
Auburn University, Auburn, AL 36849,
U.S.A. }

\date{}
\maketitle

\begin{abstract}
In the current paper, we consider the following parabolic-parabolic chemotaxis system with logistic source on $\mathbb{R}^{N}$,
\begin{equation}\label{abstract-eq}
\begin{cases}
u_{t}=\Delta u - \chi\nabla\cdot  ( u\nabla v) + u(a-bu),\quad  x\in\R^N,\,\,\, t>0 \\
{v_t}=\Delta v -\lambda v+\mu u,\quad x\in\R^N,\,\,\, t>0.
\end{cases}
\end{equation}
where $\chi, \ a,\  b,\ \lambda,\ \mu$ are positive constants and $N$ is a positive integer.  We investigate the persistence and convergence in \eqref{abstract-eq}. To this end, we
first prove, under the assumption  $b>\frac{N\chi\mu}{4}$, the  global existence of a unique classical solution $(u(x,t;u_0, v_0),v(x,t;u_0, v_0))$ of \eqref{abstract-eq} with $u(x,0;u_0, v_0)=u_0(x)$ and $v(x,0;u_0, v_0)=v_0(x)$ for every nonnegative, bounded, and uniformly continuous function $u_0(x)$, and every nonnegative, bounded, uniformly continuous, and differentiable function $v_0(x)$. Next, under the same assumption $b>\frac{N\chi\mu}{4}$, we show that persistence phenomena occurs, that is, any globally defined bounded positive classical solution with strictly positive initial function $u_0$ is bounded below by a positive constant independent of $(u_0, v_0)$ when time is large.  Finally, we discuss the asymptotic behavior of the global classical solution with strictly positive initial function $u_0$. We show that there is $K=K(a,\lambda,N)>\frac{N}{4}$ such that if $b>K \chi\mu$ and $\lambda\geq \frac{a}{2}$, then for every strictly positive initial function $u_0(\cdot)$, it holds that
$$\lim_{t\to\infty}\big[\|u(x,t;u_0, v_0)-\frac{a}{b}\|_{\infty}+\|v(x,t;u_0, v_0)-\frac{\mu}{\lambda}\frac{a}{b}\|_{\infty}\big]=0.$$
\end{abstract}



\section{Introduction and the Statements of Main results}

The current paper is devoted to the study of the asymptotic dynamics of the following parabolic-parabolic chemotaxis model with logistic source on $\R^N$:
\begin{equation}\label{Main-Eq}
\begin{cases}
u_{t}=\Delta u - \chi\nabla\cdot  (u\nabla v) + u(a-bu),\quad  x\in\R^N, \\
{v_t}=\Delta v -\lambda v+\mu u,\quad x\in\R^N,
\end{cases}
\end{equation}
%\textbf{(W.S. changed $\tau$ to $1$. make proper modifications due to this change, put the reviews on the case $\tau\not =1$ in the subsection of remarks)}
where $u(x, t)$and $v(x, t)$ denote the population densities of some mobile species and chemical
substance, respectively; $\chi>0$ represents the chemotactic sensitivity effect on the mobile species; %$\tau>0$ is a positive constant related to the diffusion rate of the chemical substance;
 $\lambda>0$ represents the degradation rate of the chemical substance and $\mu$ is the rate at which the mobile species produces the chemical substance. The reaction term $u(a-bu)$ in the first equation of \eqref{Main-Eq} is referred to as logistic source describing the local dynamics of the mobile species.

The origin of  chemotaxis models was introduced by Keller and Segel  at the beginning of 1970s in their works \cite{KeSe1}, \cite{KeSe2} to describe the aggregation of the slime mold {\it Dyctyostelium discoideum}. Chemotaxis describes the oriented movements of biological cells and organisms in response to chemical gradient which they may produce by themselves over time and is crucial for many aspects of behaviour such as the location of food sources, avoidance of predators and
attracting mates, slime mold aggregation, tumor angiogenesis, embryo development, primitive streak formation and etc.(see \cite{KJPainter}). A lot of literature is concerned with mathematical analysis of various chemotaxis models. The reader is referred to \cite{BBTW, HiPa, Hor} and the references therein for some detailed introduction into the mathematics of chemotaxis models.

Consider chemotaxis models. Central problems include the existence  of non-negative classical/weak solutions which are globally defined in time or blow up at a finite time and  the asymptotic  behavior of globally defined solutions such as persistence and convergence as time goes to infinity; etc. Considerable progress has been made in  the analysis of various chemotaxis models towards these central problems on both bounded and unbounded domains. For example, consider the following chemotaxis model on bounded domain with Neumann boundary condition
\begin{equation}\label{IntroEq-0}
\begin{cases}
u_{t}=\Delta u-\chi\nabla\cdot (u\nabla v) + f(u,v),\quad   x\in\Omega, \\
v_t=\Delta v -v+u,\quad  x\in\Omega,\\
\frac{\p u}{\p n}=\frac{\p v}{\p n}=0,\quad x\in \p \Omega.
\end{cases}
\end{equation}
Chemotaxis model \eqref{IntroEq-0} is the so-called minimal model when $f\equiv 0$. It is known that finite time blow up may occur for the minimal model. For example, when $\Omega$ is a ball in $\R^N$ with $N\geq 3$, then for all $M>0$ there exists positive initial data $(u_0,v_0)\in C(\bar\Omega)\times W^{1,\infty}(\Omega)$ with $\int_{\Omega}u_0=M$ such that the corresponding solution blows up in finite time
(see \cite{win_JMPA}).
 When $\Omega$ is a convex bounded domain with smooth boundary, $f(u,v)=u(a-bu)$ where $a$ and $b$ are positive constants,  and $\frac{b}{\chi}$ is sufficiently large,  it is shown in \cite{win_JDE2014} that  unique global classical solution of \eqref{IntroEq-0} exists for  every positive initial data $(u_{0},v_0)\in C^{0}(\overline{\Omega})\times W^{1,\infty}(\Omega)$ and that the constant solution $(\frac{a}{b},\frac{a}{b})$ is asymptotically stable in the sense that
 \[
 \lim_{t\to\infty} \big[ \|u(\cdot,t;u_0,v_0)-\frac{a}{b}\|_{L^\infty(\Omega)}+\|v(\cdot,t;u_0,v_0)-\frac{a}{b}\|_{L^\infty(\Omega)}\big]=0.
 \]
However, when $b$ is not large relative to $\chi$,  numerical evidence shows that even in the spatially one-dimensional setting solutions may exhibit chaotic behavior  (see \cite{PaHi}).  Also a phenomenon suggested by the numerical simulations
in \cite{PaHi} consists in the ability of \eqref{IntroEq-0} to enforce asymptotic smallness of the cell population
density, undistinguishable from extinction, in large spatial regions (see e.g. Fig. 7(d) in \cite{PaHi}).   In \cite{TaWi},  the authors
proved that any such extinction phenomenon must be localized in space, and that the population
as a whole always persists, which is called persistence of mass in \cite{TaWi}. 
Recently, Issa and Shen \cite{IsSh}  proved  the pointwise persistence phenomena, that is, any globally defined positive solution is bounded below by a positive constant independent of its initial data,
 which implies that the
cell population may become very small at some time and some location, but it persists at any
location eventually.  For other related works on \eqref{IntroEq-0}, we refer the readers to \cite{HorWan, lankeit_eventual,  LiMu,   NAGAI_SENBA_YOSHIDA, OTYM2002,OsYa, win_jde}  %\textbf{(W.S. make sure these papers deal with paraboic-parabolic models on bounded domains)}
and references therein.

When the second equation of \eqref{IntroEq-0} is replaced by $0=\Delta v -v+u,\, x\in\Omega$, it becomes
\begin{equation}\label{IntroEq-0-0}
\begin{cases}
u_{t}=\Delta u-\chi\nabla\cdot (u\nabla v) + f(u,v),\quad   x\in\Omega, \\
0=\Delta v -v+u,\quad  x\in\Omega,\\
\frac{\p u}{\p n}=\frac{\p v}{\p n}=0,\quad x\in \p \Omega.
\end{cases}
\end{equation}
 \eqref{IntroEq-0-0} is referred to as the parabolic-elliptic chemotaxis model which models the situation in which the chemical substance diffuses very quickly. Global existence, finite blow up and asymptotic behavior of parabolic-elliptic chemotaxis model has also been studied in many papers. For example, when $f(u,v)=u(a-bu)$ where $a$ and $b$ are positive constants, it is proved in \cite{TeWi} that unique bounded global classical solution $(u(x,t;u_0),v(x,t;u_0))$ with
 $u(x,0;u_0)=u_0(x)$ exists for any sufficiently smooth positive initial function $u_0$, under the assumption that either $N\le 2$ or $b>\frac{N-2}{N}\chi$,  and that in presence of certain  sufficiently strong
logistic damping, the constant solution $(\frac{a}{b}, \frac{a}{b})$ is asymptotically stable.
For other studies of parabolic-elliptic chemotaxis models on bounded domains, we refer the readers to \cite{DiNa, GaSaTe, lankeit_exceed, WaMuZh, Win, win_JMAA_veryweak, win_JNLS, YoYo, ZhMuHuTi} %\textbf{(W.S. make sure these papers deal with parabolic-elliptic models)}
and the references therein.

{There are also several studies of  chemotaxis models on the whole space.
  For example, consider \eqref{Main-Eq} when $f(u,v)=0$.}  It is possible for a non-negative solution  in $\R^N$ ($N\geq 2$) to blow up in finite time (see  \cite{CsPj}). It was shown in \cite{NtSrUm} that the unique solution exists globally in time and bounded under some conditions for initial data. Moreover, every bounded solution decays to $0$ as $t\to\infty$ and behaves like the heat kernel with the self-similarity (see \cite{NtYt} for the asymptotic profiles of bounded solution in the case $N=1$).

\smallskip

{When the second equation in \eqref{Main-Eq}  being replaced by $0=\Delta v -\lambda v+\mu u$,}
it becomes
\begin{equation}\label{Main-Eq-0}
\begin{cases}
u_{t}=\Delta u - \chi\nabla\cdot  (u\nabla v) + u(a-bu),\quad  x\in\R^N, \\
0=\Delta v -\lambda v+\mu u,\quad x\in\R^N.
\end{cases}
\end{equation}
 Some studies of \eqref{Main-Eq-0} are also carried out.
For example, in the  case of $a=b=0$,  it is known that blow-up occurs if
either N=2 and the total initial population mass is large enough, or $N\ge 3$ (see  \cite{BBTW}, \cite{DiNaRa} and  references therein).  When $a$ and $b$ are positive constants, it is shown for the case $\lambda=\mu=1$ in \cite{SaSh1} that if $b>\chi$, then there exists a unique bounded global classical solution for any nonnegative initial $u_0\in C_{\rm unif}^b(\R^N)$,  and that if $b>2\chi$, then for any strictly positive initial $\inf_{x\in\R^N}u_0(x)>0$, the unique global classical solution $(u(x, t;u_0), v(x, t;u_0))$ with $u(x,0; u_0)=u_0(x)$ converges to constant solution $(\frac{a}{b}, \frac{a}{b})$ as time goes to infinity. For the persistence of globally defined classical solution with strictly positive initial function, we refer the readers to \cite{SaSh3}.
For the asymptotic behavior of nonempty compact supported initials and front like initials on the whole space $\R^N$, we refer the readers to  \cite{SaSh2} and \cite{SaShXu}.

There is not much study of the asymptotic dynamics in \eqref{Main-Eq}. 
The objective of the current paper is to investigate the persistence of global classical solution of \eqref{Main-Eq} with strictly positive initial data and the asymptotic behavior of global classical solution of \eqref{Main-Eq} with strictly positive initial data.

In the rest of this introduction, we introduce the notations and definitions, and state the main results of this paper.

\subsection{Notations and statements of the main results}

In order to state our main results, we first introduce some notations and definitions.
%For every $x=(x_1,x_2,\cdots,x_N)\in\R^N$, let $|x|_{\infty}=\max\{|x_i|, \ |\ i=1,\cdots,N\}$ and $|x|=\sqrt{|x_1|^2+\cdots+|x_N|^2}$.
Let
\begin{equation*}
\label{unif-cont-space}
C_{\rm unif}^b(\R^N)=\{u\in C(\R^N)\,|\, u(x)\,\,\text{is uniformly continuous in}\,\, x\in\R^N\,\, {\rm and}\,\, \sup_{x\in\R^N}|u(x)|<\infty\}
\end{equation*}
equipped with the norm $\|u\|_\infty=\sup_{x\in\R^N}|u(x)|$, and
$$
C_{\rm unif}^{b,1}=\{u\in C_{\rm unif}^b(\R^N)\,|\, \p_{x_i}u\in C_{\rm unif}^b(\R^N),\,\, i=1,2,\cdots, N\}
$$
equipped with the norm $\|u\|_{C_{\rm unif}^{b, 1}}=\|u\|_\infty+\sum_{i=1}^N \|\p_{x_i}u\|_\infty$ and
$$
C_{\rm unif}^{b,2}=\{u\in C_{\rm unif}^{b, 1}(\R^N)\,|\, \p_{x_i x_j}u\in C_{\rm unif}^b(\R^N),\,\, i,j=1,2,\cdots, N\}.
$$
For given $0<\nu<1$, let
\begin{equation}
\label{holder-cont-space}
C^{b,\nu}_{\rm unif}(\R^N)=\{u\in C_{\rm unif}^b(\R^N)\,|\, \sup_{x,y\in\R^N,x\not =y}\frac{|u(x)-u(y)|}{|x-y|^\nu}<\infty\}
\end{equation}
with the norm $\|u\|_{\infty,\nu}=\sup_{x\in\R^N}|u(x)|+\sup_{x,y\in\R^N,x\not =y}\frac{|u(x)-u(y)|}{|x-y|^\nu}$. Hence $C^{b, 0}_{\rm unif}(\R^N)=C^{b}_{\rm unif}(\R^N)$.
For $0<\theta<1$, let
\begin{align*}
&C^{\theta}((t_1,t_2),C_{\rm unif}^{b,\nu}(\R^N))\\
&=\{ u(\cdot)\in C((t_1,t_2),C_{\rm unif}^{b,\nu}(\R^N))\,|\, u(t)\,\, \text{is locally H\"{o}lder continuous with exponent}\,\, \theta\}.
\end{align*}

We call $(u(x,t),v(x,t))$ a {\it  classical solution} of \eqref{Main-Eq} on
$ [0,T)$ if  {$ u,v \in C(\R^N\times [0,T))\cap  C^{2,1}(\R^N\times (0,T))$} and satisfies \eqref{Main-Eq} for
$(x,t)\in\R^N\times (0,T)$ in the classical sense. A classical solution $(u(x,t),v(x,t))$  of \eqref{Main-Eq} on
$ [0,T)$ is called {\it non-negative} if $u(x,t)\ge 0$ and $v(x,t)\ge 0$ for all $(x,t)\in\R^N\times [0,T)$.  A {\it global classical solution}   of \eqref{Main-Eq} is a classical solution  on
$ [0,\infty)$. Note that, due to biological interpretations, only non-negative classical solutions will be of interest.

\smallskip

The first theorem is on the local existence of a unique classical solution with initial function $(u_0,v_0)\in C_{\rm unif}^{b}(\R^N)\times C_{\rm unif}^{b,1}(\R^N)$.

\begin{tm} \label{Local-thm1}
For any $u_0 \in C_{\rm unif}^{b}(\R^N)$, $v_0 \in C_{\rm unif}^{b,1}(\R^N)$ with  $u_0 \geq 0$, $v_0 \geq 0$, there exists $T_{\max}:=T_{\max}(u_0, v_0) \in (0,\infty]$  such that \eqref{Main-Eq}  has a unique non-negative classical solution $(u(x,t;u_0, v_0)$, $v(x,t;u_0, v_0))$ on $[0,T_{\max})$ satisfying that $\lim_{t\to 0^{+}}u(\cdot,t;u_0, v_0)=u_0$ in the
$C_{\rm unif}^b(\R^N)$-norm and
$\lim_{t\to 0^{+}}v(\cdot,t;u_0, v_0)=v_0$ in the
$C_{\rm unif}^{b,1}(\R^N)$-norm,
\begin{equation}
\label{local-1-eq1}
u(\cdot,\cdot;u_0, v_0) \in C([0, T_{\max} ), C_{\rm unif}^b(\R^N) )\cap C^1((0,T_{\max}),C_{\rm unif}^b(\R^N)),
\end{equation}
\begin{equation}
\label{local-1-eq1-1}
v(\cdot,\cdot;u_0, v_0) \in C([0, T_{\max} ), C_{\rm unif}^{b,1}(\R^N) )\cap C^1((0,T_{\max}),C_{\rm unif}^{b,1}(\R^N)),
\end{equation}
\begin{equation}
\label{local-1-eq2}
u(\cdot,\cdot;u_0, v_0),\, \partial_{x_i} u(\cdot,\cdot;u_0, v_0),\, \partial^2_{x_i x_j} u(\cdot,\cdot;u_0, v_0),\, \partial_t u(\cdot,\cdot;u_0, v_0)\in C^\theta((0,T_{\max}),C^{b,\nu}_{\rm unif}(\R^N))
\end{equation}
\begin{equation}
\label{local-1-eq2-1}
v(\cdot,\cdot;u_0, v_0),\,\, \partial_{x_i} v(\cdot,\cdot;u_0, v_0),\,\, \partial^2_{x_i x_j} v(\cdot,\cdot;u_0, v_0),\,\, \partial_t v(\cdot,\cdot;u_0, v_0)\in C^\theta((0,T_{\max}),C^{b,\nu}_{\rm unif}(\R^N))
\end{equation}
for all $i,j=1,2,\cdots,N$, $0<\theta\ll 1$, and  $0<\nu\ll 1$.
Moreover, if $T_{\max}< \infty,$ then
$\lim_{t \to T_{\max}} \big(  \left\| u(\cdot,t;u_0, v_0) \right\|_{\infty}+ \| v(\cdot,t;u_0, v_0) \|_{C_{\rm unif}^{b,1}(\R^N)} \big)=\infty.$
\end{tm}

The second theorem is on the global existence of the classical solution with initial function $(u_0,v_0)\in C_{\rm unif}^{b}(\R^N)\times C_{\rm unif}^{b,1}(\R^N)$.

\begin{tm} \label{Main-thm1}
Suppose that $b>\frac{N\mu\chi}{4}$. Then for every $u_0 \in C_{\rm unif}^{b}(\R^N)$, $v_0 \in C_{\rm unif}^{b,1}(\R^N)$ with  $u_0 \geq 0$, $v_0 \geq 0$,  \eqref{Main-Eq} has a unique bounded global classical solution $(u(x,t;u_0, v_0)$, $v(x,t;u_0, v_0))$ {and
\begin{equation}\label{u-bdd-1}
\limsup_{t\to\infty}\|u(\cdot,t; u_0, v_0)\|_{\infty}\leq \frac{(2\lambda+a)^2}{2\lambda(4b-N\mu\chi)}.
\end{equation}
Moreover, if $\lambda\geq \frac{a}{2}$, then
\begin{equation}\label{u-bdd-2}
\limsup_{t\to\infty}\|u(\cdot,t; u_0, v_0)\|_{\infty}\leq \frac{4a}{4b-N\mu\chi}.
\end{equation}

} %\textbf{(W.S. also try to prove the case $\tau\not =1$)}
\end{tm}

%The second theorem is on the global existence of classical solutions with initial function $(u_0,v_0)\in L^{p}(\R^N)\times %W^{1,p}(\R^N)$.

%\begin{tm}\label{Main-thm2}
%Let $N$ be a positive integer and $p$ be a positive real number with $p>N$ and $p\geq 2$.
%Then there exists a positive constant $C_{\mu, \lambda, \tau, p, \chi}$ such that for every nonnegative initial data $u_{0}\in L^{p}(\R^{N})$, $v_0\in W^{1,p}(\R^N)$, if $b\geq \frac{C_{\mu, \lambda, \tau, p, \chi}}{p}$, then \eqref{Main-Eq} has a unique  global classical  solution $(u(x,t;u_0, v_0),v(x,t;u_0, v_0))$.
%\end{tm}

The third theorem is on the persistence of the global classical solution with strictly positive initial $u_0$.

\begin{tm}\label{Main-thm3}
Suppose that $b>\frac{N\mu\chi}{4}$, then there exist $m>0$ and $M>0$ such that
for any $u_0\in C_{\rm unif}^{b}(\R^N)$, $v_0\in C_{\rm unif}^{b,1}(\R^N)$ with $\inf_{x\in\R^N}u_0>0$ and $v_0\geq 0$, { there is $T(u_0,v_0)$ such that }
$$
m\leq u(x,t; u_0,v_0)\leq M \quad \forall\, x\in\R^N,\,\ t\geq T(u_0,v_0).
$$
\end{tm}
%\textbf{(W.S. the statement in the above theorem is stronger than before)}

The last theorem is on the asymptotic behavior of the global classical solution with strictly positive initial $u_0$.

\begin{tm}\label{Main-thm4}
There exists $K=k(a,\lambda,N)>\frac{N}{4}$ such that if  $b>K\chi\mu$ and $\lambda\ge \frac{a}{2}$,  then the unique bounded global classical solution $(u(x,t;u_0, v_0),v(x,t;u_0, v_0))$ of \eqref{Main-Eq} with $u_{0}\in C_{\rm unif}^{b}(\R^N)$, $v_0 \in C_{\rm unif}^{b,1}(\R^N)$ and $\inf_{x\in\R^N}u_{0}(x)>0$, $v_0\geq 0$,
satisfies that
\begin{equation}\label{u-v-decay}
\|u(\cdot,t;u_0, v_0)-\frac{a}{b}\|_{\infty} + \|v(\cdot,t;u_0, v_0)-\frac{\mu a}{\lambda b}\|_{\infty}\rightarrow 0 \ \text{as} \ t\rightarrow \infty \ \text{exponentially}.
\end{equation}
\end{tm}

\subsection{Remarks on the main results}

In this subsection, we provide the following remarks on the main results established in this paper.

\begin{itemize}

\item[1.]
Theorem \ref{Local-thm1} is on the local existence of a unique classical solution with nonnegative initial function $(u_0,v_0)\in C_{\rm unif}^{b}(\R^N)\times C_{\rm unif}^{b,1}(\R^N)$. We point out the local existence of a unique  classical solution with $(u_0,v_0)$ in some other spaces can also be proved. For example,  following the similar arguments used in the proof of Theorem \ref{Local-thm1}, the local existence of  a unique classical solution with nonnegative initial function $(u_0,v_0)\in L^{p}(\R^N)\times W^{1,p}(\R^N)$ for $p>N$ and $p\geq 2$ can be proved.
%For given $p>N$ and $\alpha\in(\frac{1}{2}, 1)$, let $X_i^{\alpha}$ be the fractional power space of $\lambda I-\Delta$ on %$X_i$, where  $X_1=L^{P}(\R^N)$, $X_2=W^{1,p}(\R^N)$, it can be proved by using semigroup theory there is a positive number %$T_{\max}=T_{\max}^{\alpha}(u_0, v_0)\in(0, \infty]$ such that \eqref{Main-Eq} has a unique non-negative solution %$(u(x,t;u_{0}, v_{0}),v(x,t;u_{0}, v_{0}))$ on $[0,T_{\max} )$ for any nonnegative initial data $(u_0,v_0)\in X_1^\alpha\times %X_2^\alpha$. Note that $X_1^{\alpha}\subset C_{\rm unif}^{b}(\R^N)$ for $p>N$ and $\alpha\in(\frac{1}{2}, 1)$
%and $X_2^\alpha \subset C_{\rm unif}^{b,1}(\R^N)$. The local existence of classical solution with initial function $(u_0, %v_0)\in X_1^{\alpha}\times X_2^\alpha$ is then guaranteed by Theorem \ref{Local-thm1}. However, the classical solution with %$(u_0, v_0)\in X_1^{\alpha}\times X_2^\alpha$ has more regularity, such as
%$$u(\cdot,\cdot;u_0, v_0)\in C([0, T_{\max} ), X_1^{\alpha}) \cap C^1((0,T_{\max}),L^p(\R^N))\cap C((0, T_{\max}), %X_1^{\beta}),
% $$
% and
% $$ v(\cdot,\cdot;u_0, v_0)\in C([0, T_{\max}), X_2^{\alpha}) \cap C^1((0,T_{\max}),W^{1,p}(\R^N))\cap C((0, T_{\max} ), %X_2^{\beta})$$
% for all $0\leq \beta<1$,  and
% $$\lim_{t\rightarrow 0^{+}}u(\cdot,t;u_{0},v_{0})=u_{0}(\cdot)\,\, \,\text{in}\,\,\, X_1^{\alpha}\quad {\rm and}\quad
% \quad \lim_{t\to 0^{+}}v(\cdot,t;u_0, v_0)=v_0\,\,\, \text{in}\,\,\, X_2^{\alpha},
%$$
%which  are not included in Theorem \ref{Local-thm1}.

\item[2.]
As it is mentioned in the above, consider chemotaxis model \eqref{IntroEq-0} on convex bounded domain with Neumann boundary condition and with $f(u,v)=u(1-bu)$ and $\frac{b}{\chi}$ being  sufficiently large, Winkler \cite{win_JDE2014} proved the global existence of classical solution for every nonnegative initial data $(u_{0},v_0)\in C^{0}(\overline{\Omega})\times W^{1,\infty}(\Omega)$ and the global asymptotic stability of the constant solution $(\frac{1}{b},\frac{1}{b})$.
%in the sense that
% \[
% \lim_{t\to\infty} \big[ %\|u(\cdot;t;u_0)-\frac{1}{b}\|_{L^\infty(\Omega)}+\|v(\cdot,t;u_0)-\frac{1}{b}\|_{L^\infty(\Omega)}\big]=0.
% \]
Theorem \ref{Main-thm1} and Theorem \ref{Main-thm4} stated in the above extend the results in \cite{win_JDE2014} on the global existence and global asymptotical stability of the constant solution  for parabolic-parabolic chemotaxis systems on bounded domains to the whole space. Biologically, the  conditions $b>\frac{N\chi\mu}{4}$ and $b>K\chi\mu$ in Theorems \ref{Main-thm1} and \ref{Main-thm4} indicate that the logistic damping $b$ is large relative to the product of the chemotaxis sensitivity
$\chi$ and the production rate $\mu$ at which the mobile species produces the chemical substance.
The condition $\lambda\ge \frac{a}{2}$ in Theorem \ref{Main-thm4} indicates that the degradation rate of the chemical substance
is large relative to the intrinsic growth rate of the mobile species.

\item[3.]
In \cite{IsSh}, Issa and Shen studied  pointwise persistence in \eqref{IntroEq-0} with $f$ being local as well as nonlocal time and space dependent logistic source. Under certain conditions, they proved that any globally defined positive solution is bounded below by a positive constant independent of its initial data (see \cite[Theorem 1.2]{IsSh}). Our persistence result Theorem \ref{Main-thm3} extends \cite[Theorem 1.2]{IsSh} for parabolic-parabolic chemotaxis systems on bounded domains to the whole space.
Due to the unboundedness of the underlying environment, it is highly nontrivial to  adopt
the arguments in  \cite[Theorem 1.2]{IsSh} to prove Theorem \ref{Main-thm3}.
For given $x_0\in\R^N$ and $L>0$, let
$$
B_L(x_0)=\{x\in\R^N\,|\, |x-x_0|< L\}.
$$
 The main idea to  prove  Theorem \ref{Main-thm3} is to prove that  there are $L^*\gg 1$,
 $\epsilon^*>0$, $\delta^*>0$,  $T^{**}>T^*$, and $M^*>0$  such that   for any  non-negative $(u_0,v_0)\in C_{\rm unif}^b(\R^N)\times C_{\rm unif}^{b,1}(\R^N)$
with $\inf_{x\in\R^N}u_0(x)>0$, the following hold:

 \item[] i)  there is $T^*_0(u_0,v_0)>0$  such that for any $x_0\in\R^N$,  if
 $$
\sup_{x\in{B}_{L^*}(x_0)} u(x,t;u_0,v_0)\le \epsilon^*\quad \forall\, \, T^*_0(u_0,v_0)\le t_1\le t< t_2,
$$
then
\begin{equation*}
\sup_{x\in{B}_{\frac{L^*}{2}}(x_0)}\max\{ v(x,t;u_0,v_0),|\p_{x_i} v(x,t;u_0,v_0)|\}\leq M^*\epsilon^*\quad \forall\,\ t_1+T^*\le t< t_2,\,\, i=1,2,\cdots,N
\end{equation*}
 (that is, if $u(x,t;u_0,v_0)$ is smaller than $\epsilon^*$  on  the space  ball $B_{L^*}(x_0)$ for
 $t\in [t_1,t_2)$,   then $v(x,t;u_0,v_0)$ {and $|\p_{x_i} v(x,t;u_0,v_0)|$ are} smaller than $M^*\epsilon^*$  on the space ball $B_{\frac{L^*}{2}}(x_0)$
 for $t\in [t_1+T^*,t_2)$);

 \item[] ii)  if
$$
\sup_{x\in B_{L^*}(x_0)}u(x,t_0;u_0,v_0)\ge \epsilon^*\quad {\rm and}\quad t_0\ge T_0^*(u_0,v_0),
$$
then
\begin{equation}
\label{persistence-eq3}
\inf_{x\in B_{L^*}(x_0)} u(x,t;u_0,v_0)\ge \delta^* \quad \forall\, t_0\le t\le t_0+T^{**}
\end{equation}
(that is, if $\sup_{x\in B_{L^*}(x_0)}u(x,t_0;u_0,v_0)\ge \epsilon^*$ at some $t_0\ge T_0^*(u_0,v_0)$,
  then $u(x,t;u_0,v_0)$ is larger than $\delta^*$ on the space  ball $B_{L^*}(x_0)$ for $t$ in the interval $[t_0,t_0+T^{**}]$);

 \item[] iii)  if
$$
\begin{cases}
\sup_{x\in B_{L^*}(x_0)}u(x,t_1;u_0,v_0)= \epsilon^*\quad \text{for some}\,\, t_1\ge T_0^*(u_0,v_0)\cr
\sup_{x\in B_{L^*}(x_0)}u(x,t;u_0,v_0)\le  \epsilon^*\quad {\rm for}\,\, t_1< t< t_2\le \infty,
\end{cases}
$$
  then
\begin{equation}
\label{persistence-eq2}
\inf_{x\in B_{L^*}(x_0)} u(x,t;u_0,v_0)\ge \delta^*\quad \forall\,   t_1\le t< t_2
 \end{equation}
(that is, if $\sup_{x\in B_{L^*}(x_0)}u(x,t;u_0,v_0)$ equals $\epsilon^*$ at $t=t_1$ and is less than or equal to $\epsilon^*$ for $t_1<t<t_2$, then  $u(x,t;u_0,v_0)$ is larger than $\delta^*$ on the space  ball $B_{L^*}(x_0)$ for $t$ in the time interval $[t_1, t_2)$);

\item[] iv) there is $T_0^{**}(u_0,v_0)$ such that
  if
$$
\sup_{x\in B_{L^*}(x_0)}u(x,t;u_0,v_0)<\epsilon^*\quad {\rm for}\quad T_0^*(u_0,v_0)\le t <t^*,
$$
then
$$
t^*-T_0^*(u_0,v_0)\le  T_0^{**}(u_0,v_0)
$$
(that is, $u(x,t;u_0,v_0)$ cannot be smaller than $\epsilon^*$  on  the space ball $B_{L^*}(x_0)$ for  a very long  time interval  starting at $T_0^*(u_0,v_0)$). It then follows that
$$
\inf_{x\in\R^N} u(x,t;u_0,v_0)\ge m:=\delta^*\quad \forall\,\ t\ge T_{0}^{*}(u_0,v_0)+T_{0}^{**}(u_0,v_0).
 $$

We believe that these techniques will also play a crucial rule in the study of asymptotic spreading of non-negative bounded global classical solutions.

\item[4.] Theorem \ref{Main-thm4} does not give an explicit expression on $K$, but it has the following property. According to the proof of
Theorem \ref{Main-thm4}, $K=\frac{N}{4\theta_0}$,  where $\theta_0\in (0,1)$ is the largest number such that
\begin{equation}\label{eq-1-00}
\frac{2C_2\theta_0}{(1-\theta_0)^2 a }\le \frac{1}{6}
\quad {\rm and}\quad
\frac{8C \lambda^{-\frac{1}{2}}a^{\frac{1}{2}}\pi \theta_0}{N(1-\theta_0)}\le \frac{1}{12}
\end{equation}
hold simultaneously with $C_2=\max\{C\lambda^{\gamma-\beta-\frac{3}{2}}a^{\beta+\frac{3}{2}}\sqrt{\pi}N^{-2}+C\lambda^{\gamma-\beta-1}a^{\beta+1}N^{-1}, a\}$,  $C$  here  as well as in \eqref{eq-1-00}  is a generic constant and  $\beta$ and $\gamma$ are such that $\gamma\in(1,\frac{3}{2})$ and $\gamma-1<\beta<\frac{1}{2}$.  It can then be verified directly that
for fixed $a$ and $N$, $K$ is bounded in $\lambda\ge \frac{a}{2}$ and
  $K\to \frac{N}{0.28}$ as $\lambda\to \infty$.

\item[5.]
Consider the following general parabolic-parabolic chemotaxis model on bounded domain with Neumann boundary condition
\begin{equation}\label{IntroEq-tau}
\begin{cases}
u_{t}=\Delta u-\chi\nabla\cdot (u\nabla v) + f(u,v),\quad   x\in\Omega, \\
\tau v_t=\Delta v -v+u,\quad  x\in\Omega,\\
\frac{\p u}{\p n}=\frac{\p v}{\p n}=0,\quad x\in \p \Omega.
\end{cases}
\end{equation}
where $\tau>0$ is a positive constant related to the diffusion rate of the chemical substance.
Winkler \cite{win_CPDE2010} considered the system \eqref{IntroEq-tau} in a smooth bounded convex domain $\Omega\subset \R^N$ with $\tau>0$ ($\tau$ is not necessarily 1), $f(u,v)=u(a-bu)$ where $a$ and $b$ are positive constants, $\chi\in\R$ and established the global existence and
boundedness of non-negative classical solutions of system \eqref{IntroEq-tau} provided that $b$ is large enough.
In \cite{ZhLiBaZo}, Zheng, Li, Bao and Zou extended Winkler’s global existence result to bounded domains
(not necessarily convex) of $\R^N$ for $\chi>0$ and proved that if the logistic dampening $b>\frac{(N-2)_{+}}{N}\chi[C_{\frac{N}{2}+1}]^{\frac{1}{\frac{N}{2}+1}}$, where $C_{\frac{N}{2}+1}$ is a positive constant which is corresponding to the maximal Sobolev regularity, then \eqref{IntroEq-tau} admits a unique, smooth,and bounded global
non-negative solution. Recently, Issa and Shen \cite{IsSh} extended the
global existence results obtained in both \cite{win_CPDE2010} and \cite{ZhLiBaZo} to the general full chemotaxis model \eqref{IntroEq-tau} with $f$ being local as well as nonlocal time and space dependent logistic source. The global existence of classical solutions of \eqref{IntroEq-tau} with $\tau\not =1$ and $f$ being logistic source on the whole space $\R^N$ will be studied somewhere else.

\end{itemize}

The rest of the paper is organized as follows: In section 2, we present some preliminary materials that will be needed in the proofs of our main results. In section 3, we study the local and global existence of the classical solution of \eqref{Main-Eq} with given initial function and prove Theorems \ref{Local-thm1} and \ref{Main-thm1}. In section 4, we explore the persistence of the global classical solution of \eqref{Main-Eq} with strictly positive initial $u_0$ and prove Theorem \ref{Main-thm3}. The last section is devoted to discuss the asymptotic behavior of global classical solutions and prove Theorem \ref{Main-thm4}.

\section{Preliminaries}

%\textbf{(W.S. Present those which will be needed in the following sections)}

In this section, we present several lemmas which will be used often in the later sections. The reader is referred to \cite{Dan Henry}, \cite{A. Pazy} for the details.

 Throughout this paper,  $\{e^{t(\Delta-\sigma I)}\}_{t>0}$ where $\sigma>0$ denotes the analytic  semigroup generated by $\Delta-\sigma I $ on $X:=C_{\rm unif}^b(\R^N)$, unless specified otherwise.
Then we have
\begin{equation}\label{Lp-Lq-1}
 \|e^{t(\Delta-\sigma I)} u \|_{\infty}\leq e^{-\sigma t}\|u\|_{\infty},
\end{equation}
\begin{equation}\label{Lp-Lq-2}
 \|\nabla e^{t(\Delta-\sigma I)} u \|_{\infty}\leq  C_{N} t^{-\frac{1}{2}}e^{-\sigma t}\|u\|_{\infty},
\end{equation}
\begin{equation}\label{Lp-Lq-3}
 \|(\sigma I-\Delta)^{\alpha} e^{t(\Delta-\sigma I)}u \|_{\infty}\leq C_{\alpha} t^{-\alpha}e^{-\sigma t}\|u\|_{\infty}
\end{equation}
 for every $t>0$ and $\alpha\geq 0$.  In fact,
 \eqref{Lp-Lq-1} and \eqref{Lp-Lq-2} follow directly from the following equation,
 \begin{equation}
\label{semigroup-eq}
(e^{t(\Delta-\sigma I)} u)(x)= \int_{\R^{N}}e^{-\sigma t}\frac{1}{(4\pi t)^{\frac{N}{2}}}e^{-\frac{|x-y|^{2}}{4t}}u(y)dy
\end{equation}
for every $u\in C_{\rm unif}^b(\R^{N})$, $t>0$, $x\in\R^N$.
   \eqref{Lp-Lq-3} is a result of the combination of Theorem 1.4.3 in \cite{Dan Henry} and \eqref{Lp-Lq-1}.

\begin{lem}\label{L_Infty bound 2}
 For every  $t>0$, the operator $e^{t(\Delta-\sigma I)} \nabla\cdot $ has a unique bounded extension on $\big(C_{\rm unif}^b(\R^N)\big)^N$  satisfying
\begin{equation}\label{EqL_Infty03}
\|e^{t(\Delta-\sigma I)}\nabla \cdot u\|_{\infty}\leq\frac{N}{\sqrt{\pi}} t^{-\frac{1}{2}}e^{-\sigma t}\|u\|_{\infty} \ \ \ \forall \ u\in \big(C_{\rm unif}^{b}(\R^N)\big)^N, \ \forall \ t>0.
\end{equation}
\end{lem}

\begin{proof}
It follows from \cite[Lemma 3.2]{SaSh1}.
\end{proof}

 Note that ${\rm Dom}(\Delta-\sigma I )=C^{b,2}_{\rm unif}(\R^N)$. Let $X^{\alpha}={\rm Dom}((\sigma I-\Delta )^{\alpha})$ be the fractional power space of $\sigma I-\Delta$ on $X$ ($\alpha \in [0,1]$) equipped with graph norm $\|u\|_{X^{\alpha}}=\|(\sigma I-\Delta)^{\alpha}u\|_{X}$.  We have
the following continuous imbedding
\begin{equation}\label{Fractional power Imbedding-1}
X^{\alpha} \hookrightarrow C^{\nu} \quad \text{if} \quad 0\leq \nu < 2\alpha
\end{equation}
(see \cite[Exercise 9 in Chapter 1]{Dan Henry}).
Furthermore, there is a constant $C_{\alpha}$ such that
\begin{equation}\label{Eq_00}
\|(e^{t(\Delta-\sigma I)}-I)u\|_{X}\leq C_{\alpha}t^{\alpha}\|u\|_{X^{\alpha}} \quad \text{for all}\ u\in X^{\alpha}.
\end{equation}
Inequality \eqref{Eq_00} comes from \cite[Theorem 1.4.3]{Dan Henry}. Note that $X^0=X$ and $X^1={\rm Dom}(\sigma I-\Delta)$.

We end this section by stating an important result that will be used in the proof of the local existence of classical solutions.

\begin{lem}(\cite[Exercise $4^{*}$, page 190]{Dan Henry})
\label{Lem2}
Assume that $a_{1},a_{2},\alpha,\beta$ are non-negative constants , with $0\leq \alpha, \ \beta<1,$ and $0<T<\infty$.  There exists a constant $M(a_{2},\beta,T)<\infty$ so that for any integrable function $u : [0,\ T]\rightarrow \R$ satisfying that
$$
0\leq u(t)\leq a_{1}t^{-\alpha} +a_{2}\int_{0}^{t}(t-s)^{-\beta}u(s)ds
$$
for a.e t in $[0, \ T],$ we have
$$
0\leq u(t)\ \leq\ \frac{a_{1} M}{1-\alpha}t^{-\alpha}, \ \ \text{a.e. on }\ 0< t< T.
$$
\end{lem}

\section{Local and global existence of classical solutions}
%\textbf{(W.S. state the local existence theorems and prove them in this section. May only state and prove the following two.)}
%\noindent *****************************

This section is devoted to the study of the local and global existence of classical solutions of \eqref{Main-Eq} with given initial functions and prove Theorems \ref{Local-thm1} and \ref{Main-thm1}. Throughout this section, unless specified otherwise,
$C$ denotes a generic  constant
independent of $u,v$ and may be different at different places.

\subsection{Proof of Theorem \ref{Local-thm1}}

In this subsection, we prove Theorem \ref{Local-thm1}. The main tools for the proof of this theorem are based on the contraction mapping theorem and the existence of classical solutions for linear parabolic equations with H\"older continuous coefficients. Throughout this subsection,  $X_1=C_{\rm unif}^b(\R^N)$, $X_2= C_{\rm unif}^{b,1}(\R^N)$, and $X_{i}^\alpha$ is the fractional power space of $\lambda I-\Delta$ acting on $X_i$, $i=1,2$ ($\alpha\in (0,1)$).

\begin{proof}[Proof of Theorem \ref{Local-thm1}]
It can be proved by properly modifying arguments of \cite[Theorem 1.1]{SaSh1}.
 For the reader's convenience, we provide
the outline of the proof.

\medskip

\noindent \textbf{(i) Existence of a mild solution}. We first prove the existence of a mild solution of \eqref{Main-Eq} with given initial function $u_0\in C_{\rm unif}^b(\R^N)$, $v_0 \in C_{\rm unif}^{b,1}(\R^N)$, that is, the existence of $(u(t),v(t))$ satisfying
\begin{equation}
\label{mild-solution-eq}
\begin{cases}
u(t)=e^{t(\Delta-\lambda I)}u_0-\chi\int_0^t e^{(t-s)(\Delta-\lambda I)}\nabla\cdot(u(s)\nabla v(s))ds\cr
\qquad \quad +\int_0^t e^{(t-s)(\Delta-\lambda I)}u(s)(a+\lambda-bu(s))ds\cr
v(t)=e^{t(\Delta-\lambda I)}v_0+\mu\int_0^t e^{(t-s)(\Delta-\lambda I)}u(s)ds.
\end{cases}
\end{equation}

 To this end, let $X=X_1\times X_2$. Fix $(u_0,v_0)\in X$. For every $T>0$ and $R>0$ satisfying $\|u_0\|_{\infty} \leq R$ and $\|v_0\|_{C_{\rm unif}^{b,1}(\R^N)} \leq R$, let
$$\mathcal{S}_{R,T}:=\left\{(u, v) \in C([0, T ], C_{\rm unif}^b(\R^N) )\times C([0, T ], C_{\rm unif}^{b,1}(\R^N) ) \ | \ \|u\|_{\infty} \leq R\,\, {\rm and}\,\
\|v\|_{C_{\rm unif}^{b,1}(\R^N)} \leq R \right\}.
$$
Note that $\mathcal{S}_{R,T}$ is a closed subset of the Banach space $C([0, T ], C_{\rm unif}^b(\R^N) )\times C([0, T ], C_{\rm unif}^{b,1}(\R^N) )  $ with the norm
$\|(u,v)\|_{\mathcal{S}_{R,T} }=\sup_{0\leq t\leq T}\|u(t)\|_{\infty}+\sup_{0\leq t\leq T}\|v\|_{C_{\rm unif}^{b,1}(\R^N)}$.
We prove the existence of a mild solutions via five claims.

\medskip

\noindent {\bf Claim 1.} {\it For any $(u, v)\in \mathcal{S}_{R,T}$ and $t\in [0,T]$,
 $\Phi(u, v)(t)= (\Phi_1(u, v)(t),\Phi_2(u, v)(t))$ is well defined {in $X$}, where
$$
\Phi_1(u,v)(t)=e^{t(\Delta-\lambda I)}u_0-\chi\int_0^t e^{(t-s)(\Delta-\lambda I)}\nabla\cdot(u(s)\nabla v(s))ds+\int_0^t e^{(t-s)(\Delta-\lambda I)}u(s)(a+\lambda-bu(s))ds,
$$
and
$$
\Phi_2(u,v)(t)=
 e^{t(\Delta-\lambda I)}v_0+\mu\int_0^t e^{(t-s)(\Delta-\lambda I)}u(s)ds.
 $$
}

Indeed, let $(u, v)\in\mathcal{S}_{R,T}$ and $0<t\leq T$ be fixed.
By the similar arguments to those of Claim 1 of \cite[Theorem 1.1]{SaSh1},
$\Phi_1(u, v)(t)$ is well defined in  $C^b_{\rm unif}(\R^N)$.

Since $v_0\in C_{\rm unif}^{b,1}(\R^N)$, we have  $e^{t(\Delta-\lambda I)}v_0$ is well defined in $C_{\rm unif}^{b,1}(\R^N)$.
For every $s_1, s_2\in[0, t)$, by \eqref{Lp-Lq-2},   we have that
\begin{align*}
&\|\nabla e^{(t-s_1)(\Delta-\lambda I)}u(s_1)-\nabla e^{(t-s_2)(\Delta-\lambda I)}u(s_2)\|_{\infty}\cr
&\leq \|\nabla e^{(t-s_1)(\Delta-\lambda I)}u(s_1)-\nabla e^{(t-s_1)(\Delta-\lambda I)}u(s_2)\|_{\infty}+\|\nabla e^{(t-s_1)(\Delta-\lambda I)}u(s_1)-\nabla e^{(t-s_2)(\Delta-\lambda I)}u(s_2)\|_{\infty}\cr
&\leq C(t-s_1)^{-\frac{1}{2}}e^{-\lambda(t-s_1)}\|u(s_1)-u(s_2))\|_{\infty}+C(t-s_2)^{-\frac{1}{2}}e^{-\lambda(t-s_2)}\|e^{(s_2-s_1)(\Delta-\lambda I)}u(s_2)-u(s_2))\|_{\infty}\cr
&\to 0 \quad {\rm as}\,\ s_1\to s_2.
\end{align*}
Thus, $[0, t)\ni s\mapsto \nabla e^{(t-s)(\Delta-\lambda I)}u(s)$ is continuous.
 By  \eqref{Lp-Lq-2} again,
$$
\|\nabla e^{(t-s)(\Delta-\lambda I)} u (s)\|_{\infty}\leq C(t-s)^{-\frac{1}{2}}e^{-\lambda (t-s)}\|u(s)\|_{\infty}\,\, \forall\, s\in [0,t)
$$
and $s\mapsto (t-s)^{-\frac{1}{2}}e^{-\lambda (t-s)}\|u(s)\|_{\infty}$ is integrable on $[0,t)$.
Hence, by the dominated convergence theorem,  $\nabla \int_0^t e^{(t-s)(\Delta-\lambda I)}u(s)ds$ exists and
$$\nabla \int_0^t e^{(t-s)(\Delta-\lambda I)}u(s)ds=\int_0^t \nabla e^{(t-s)(\Delta-\lambda I)}u(s)ds\in C^b_{\rm uinf}(\R^N).$$
Thus, $\Phi_2(u, v)(t)$ is well defined in  $C^{b,1}_{\rm unif}(\R^N)$.
Whence, the Claim 1 follows.

\medskip

\noindent {\bf Claim 2.}  {\it For every $(u, v)\in \mathcal{S}_{R,T}$, choose $0<\beta<\frac{1}{4}$ and $\frac{1}{2}<\gamma<1$ such that $\gamma+2\beta<1$. Then the function $(0, T]\ni t\to \Phi (u, v)(t)\in X^{\beta}$ is locally H\"older continuous, and $\Phi$ maps $\mathcal{S}_{R,T}$ into $C([0,T],C_{\rm unif}^b(\R^N))\times C([0,T],C_{\rm unif}^{b,1}(\R^N)) $.}

\medskip

First, by the similar arguments to those in Claim 2 of \cite[Theorem 1.1]{SaSh1},
 the function $(0, T]\ni t\to \Phi_1(u, v)(t)\in X_1^{\beta}$ is locally H\"older continuous.

Next, observe that
$$
\Phi_2(u,v)(t)=\underbrace{e^{t(\Delta-\lambda I)}v_0}_{J_0(t)}+\mu\underbrace{\int_0^t e^{(t-s)(\Delta-\lambda I)}u(s)ds}_{J_1(t)}
$$
For every $t>0$, it is clear that $J_0(t)=e^{t(\Delta-\lambda I)}v_{0}\in X_2^{\beta}$ because the semigroup $\{e^{t(\Delta-\lambda I)}\}_{t\geq 0}$ is  analytic.
Furthermore, since $X^{\gamma}_1\hookrightarrow C^1$, we have that
\begin{align}
\|J_1(t)\|_{X_2^\beta}&\le  \int_{0}^{t}\|(\lambda I-\Delta)^{\beta}e^{(t-s)(\Delta-\lambda I)}u(s) \|_{C^{b,1}_{\rm{unif}}(\R^N)}ds\nonumber \\
&\leq C\int_{0}^{t}\|(\lambda I-\Delta)^{\gamma+\beta}e^{(t-s)(\Delta-\lambda I)}u(s) \|_{\infty}ds\nonumber \\
&\leq  C\int_{0}^{t}(t-s)^{-\gamma-\beta}e^{-\lambda(t-s)}\|u(s)\|_{\infty}ds\leq  C R.
\end{align}
Since the operator $(\lambda I-\Delta)^{\beta}$ is closed, we have $J_1(t)\in X^{\beta}_2$. Hence, $\Phi_2(u,v)(t)\in X_2^{\beta}$ for every $t>0$. Therefore, $\Phi(u,v)(t)\in X^{\beta}$ for every $t>0$.

By \cite[Theorem 1.4.3]{Dan Henry}, we have
 \begin{align}\label{correct05}
\|J_{0}(t+h)-J_{0}(t)\|_{X_2^{\beta}}
&=  \|(e^{h(\Delta-\lambda I)}-I)e^{t(\Delta-\lambda I)}v_{0}\|_{X_2^{\beta}}\leq  C h^\beta\|(\lambda I-\Delta)^{\beta}e^{t(\Delta-\lambda I)}v_{0}\|_{X_2^{\beta}}\nonumber\\
&\leq C h^\beta\|(\lambda I-\Delta)^{2\beta}e^{t(\Delta-\lambda I)}v_{0}\|_{X_1^{\gamma}}
\leq  C R t^{-2\beta-\gamma}h^{\beta}.
\end{align}
Hence, $(0, T]\ni t\to J_0(t)\in X_2^\beta$ is
locally H\"older continuous. We also have
 \begin{align}\label{correct05}
&\|J_{1}(t+h)-J_{1}(t)\|_{X_2^{\beta}}\nonumber\\
&\leq \int_{0}^{t}\|(e^{h(\Delta-\lambda I)}-I)e^{(t-s)(\Delta-\lambda I)}u(s)\|_{X_2^{\beta}}ds + \int_{t}^{t+h}\|e^{(t+h-s)(\Delta-\lambda I)}u(s)\|_{X_2^{\beta}}ds \nonumber\\
&\leq C h^\beta\int_{0}^{t}\|(\lambda I-\Delta)^{2\beta}e^{(t-s)(\Delta-\lambda I)}u(s)\|_{X_2}ds + \int_{t}^{t+h}\|(\lambda I-\Delta)^{\beta} e^{(t+h-s)(\Delta-\lambda I)}u(s)\|_{X_2}ds \nonumber\\
&\leq C h^\beta\int_{0}^{t}(t-s)^{-2\beta-\gamma}e^{-\lambda(t-s)}\|u(s)\|_{\infty}ds + C \int_{t}^{t+h}(t+h-s)^{-\beta-\gamma}e^{-\lambda(t+h-s)} \|u(s)\|_{\infty}ds\nonumber\\
&\leq   CR(h^{\beta} +  h^{1-\gamma-\beta} ).
\end{align}
Hence, $(0, T]\ni t\to \Phi_2(u, v)(t)\in X_2^{\beta}$ is locally H\"older continuous. Thus, $(0, T]\ni t\to \Phi(u, v)(t)\in X^{\beta}$ is locally H\"older continuous.  It is clear that {$t\to \Phi(u,v)(t)\in X$} is continuous in $t$ at $t=0$. Claim 2 thus follows.

\medskip

\noindent {\bf Claim 3.}
 {\it For every $ R>\max\{\|u_{0}\|_{\infty}, \|v_0\|_{C^{b,1}_{\rm unif}(\R^N)}\}$, there exists $T:=T(R)$ such that  $\Phi$ maps $\mathcal{S}_{R,T}$ into itself.}

 \medskip

First, observe that for any $(u, v)\in \mathcal{S}_{R,T}$,  we have
\begin{align}\label{aux-claim3-eq1}
 \|\Phi_1(u, v)(t)\|_{\infty }
 & \leq  \|e^{t(\Delta-\lambda I)}u_{0}\|_{\infty}+\chi\int_0^t\| e^{(t-s)(\Delta-\lambda I)}\nabla\cdot ( u(s)  \nabla v(s))\|_{\infty}ds \nonumber\\
 &\,\, +(\lambda+a)\int_{0}^{t} \|e^{(t-s)(\Delta-\lambda I)}u(s)\|_{\infty}ds+b\int_{0}^{t} \|e^{(t-s)(\Delta-\lambda I)}u^{2}(s)\|_{\infty}ds \nonumber\\
 & \leq  e^{-\lambda t}\|u_{0}\|_{\infty}+\chi R^2 C\int_0^t (t-s)^{-\frac{1}{2}} e^{-\lambda (t-s)}ds \nonumber\\
 &\,\,  +(\lambda+a)R\int_{0}^{t} e^{-\lambda(t-s)}ds+bR^{2}\int_{0}^{t} e^{-\lambda(t-s)}ds \nonumber\\
 & \leq e^{-\lambda t}\|u_{0}\|_{\infty} +2\chi R^{2}Ct^{\frac{1}{2}} + R\left((1+\frac{a}{\lambda})+b\frac{R}{\lambda}\right)(1-e^{-\lambda t}).
 \end{align}

Next, for any $(u, v)\in \mathcal{S}_{R,T}$,  we have
\begin{align}\label{aux-claim3-eq5}
\|\Phi_2(u,v)(t)\|_{C_{\rm unif}^{b,1}(\R^N)}&\leq \|e^{t(\Delta-\lambda I)}v_0\|_{C_{\rm unif}^{b,1}(\R^N)}+\|\mu \int_0^t e^{(t-s)(\Delta-\lambda I)}u(s)ds\|_{C_{\rm unif}^{b,1}(\R^N)}\nonumber\\
&\leq \|e^{t(\Delta-\lambda I)}v_0\|_{\infty}+C\|\nabla e^{t(\Delta-\lambda I)}v_0\|_{\infty} \nonumber\\
&\,\,\, + \|\mu\int_0^t e^{(t-s)(\Delta-\lambda I)}u(s)ds\|_{\infty} +C\|\mu\nabla\int_0^t e^{(t-s)(\Delta-\lambda I)}u(s)ds\|_{\infty} \nonumber\\
&\le e^{-\lambda t}\|v_0\|_{\infty}+Ce^{-\lambda t}\|\nabla v_0\|_{\infty} + \mu\sup_{0\leq s\leq t}{\|u(s)\|_{\infty}}\int_0^t e^{-\lambda(t-s)}ds\nonumber\\
 &\,\,\, +\mu C\sup_{0\leq s\leq t}{\|u(s)\|_{\infty}}\int_0^t (t-s)^{-\frac{1}{2}}e^{-\lambda(t-s)}ds \nonumber\\
&\leq Ce^{-\lambda t}\|v_0\|_{C_{\rm unif}^{b,1}(\R^N)}+\frac{\mu}{\lambda}R(1-e^{- \lambda t})+\mu CRt^{\frac{1}{2}}.
\end{align}
Claim 3 then follows from \eqref{aux-claim3-eq1},  \eqref{aux-claim3-eq5},
and Claim 2.

\medskip

\noindent \textbf{Claim 4.}
 {\it $\Phi$ is a contraction map for T small and hence  has a fixed point $(u(\cdot), v(\cdot))\in \mathcal{S}_{R,T}$. Moreover, for every
$0<\beta<\frac{1}{4}$,
$\frac{1}{2}<\gamma<1$ such that $\gamma+2\beta<1$, the function $t\in(0,T]\to (u(t), v(t))\in X^{\beta}$ is locally H\"older continuous.}

For every $(u, v), (\bar u, \bar v) \in\mathcal{S}_{R,T},$ using again Lemma \ref{L_Infty bound 2}, we have
\begin{align*}
 \|\Phi_1(u,v)(t)-\Phi_1(\bar u,\bar v)(t)\|_{\infty}
&\leq  {2}\chi C\int_{0}^{t}(t-s)^{-\frac{1}{2}}e^{-\lambda(t-s)}\| u(s) -\bar u(s)\|_{\infty}\|\nabla \bar v(s)\|_{\infty}ds \nonumber\\
& \quad  + {2}\chi C\int_{0}^{t}(t-s)^{-\frac{1}{2}}e^{-\lambda(t-s)}\|u(s)\|_{\infty}\| \nabla (v(s)-\bar v(s))\|_{\infty}ds \nonumber\\
& \quad  +(\lambda+a+2Rb)\sup_{0\leq s\leq t} \|u(s)-\bar u(s)\|_{\infty}\int_{0}^{t} e^{-\lambda(t-s)}ds\nonumber\\
&\leq \left[4RC\chi t^{\frac{1}{2}}+ (\lambda+a+2Rb)t \right]\|( u, v) -(\bar u, \bar v)\|_{\mathcal{S}_{R,T}},
\end{align*}
and
\begin{align*}
& \|\Phi_2(u,v)(t)-\Phi_2(\bar u,\bar v)(t)\|_{C^{b, 1}_{\rm{unif}}(\R^N)}\nonumber\\
& \leq \mu \int_{0}^{t}{\|e^{(t-s)(\Delta-\lambda I)}(u(s)-\bar u(s))ds} \|_{\infty}ds+
\mu C\int_{0}^{t}{\|\nabla e^{(t-s)(\Delta-\lambda I)}(u(s)-\bar u(s))ds} \|_{\infty}ds
\nonumber\\
& \leq \mu \sup_{0\leq s\leq t} \|u(s)-\bar u(s)\|_{\infty} \int_{0}^{t}e^{-\lambda(t-s)}ds+
\mu C\sup_{0\leq s\leq t} \|u(s)-\bar u(s)\|_{\infty} \int_{0}^{t}(t-s)^{-\frac{1}{2}}e^{-\lambda(t-s)}ds\nonumber\\
&\leq (\mu t+2\mu Ct^{\frac{1}{2}})\|( u, v) -(\bar u, \bar v)\|_{\mathcal{S}_{R,T}}.
\end{align*}
Hence, choose T small satisfying
$$
4RC\chi t^{\frac{1}{2}}+ (\lambda+a+2Rb)t+\mu t+2\mu Ct^{\frac{1}{2}}  < 1\quad \forall\,\, t\in [0,T],
$$
we have that $\Phi$ is a contraction map. Thus there is $T>0$ and a unique function $(u, v)\in \mathcal{S}_{R,T}$ such that
$(u(t),v(t))$ satisfies \eqref{mild-solution-eq} for $t\in [0,T]$.
Moreover, by Claim 2,  for every $0<\beta<\frac{1}{4}$, $\frac{1}{2}<\gamma<1$ such that $\gamma+2\beta<1$, the function $t\in(0,T]\to (u(t), v(t))\in X^{\beta}$ is locally H\"older continuous.
Clearly, $(u(t), v(t))$ is a mild solution of \eqref{Main-Eq} on $[0,T)$. %with $\alpha=0$ and $X^0=C_{\rm unif}^b(\R^N)\times C_{\rm unif}^{b,1}(\R^N)$.

\medskip

\noindent {\bf Claim 5.} {\it There is $T_{\max}\in (0,\infty]$ such that \eqref{Main-Eq} has a mild solution $(u(\cdot), v(\cdot))$ on $[0,T_{\max})$.
Moreover, for every $0<\beta<\frac{1}{4}$,  $\frac{1}{2}<\gamma<1$ such that $\gamma+2\beta<1$,  the function  $(0,T_{\max})\ni t\mapsto (u(\cdot), v(\cdot))\in X^\beta$ is locally H\"older continuous.  If $T_{\max}<\infty$, then
$$\limsup_{t \to T_{\max}} \big(  \left\| u(\cdot,t;u_0, v_0) \right\|_{\infty}+ \| v(\cdot,t;u_0, v_0) \|_{C_{\rm unif}^{b,1}(\R^N)} \big)=\infty.
$$}

\smallskip
This claim follows the regular extension arguments.

\medskip

\noindent \textbf{(ii) Regularity and non-negativity.} We  next prove that the mild solution $(u(\cdot), v(\cdot))$ of \eqref{Main-Eq} on $[0,T_{\max})$ obtained in (i) is a non-negative classical solution of \eqref{Main-Eq} on $[0,T_{\max})$ and satisfies \eqref{local-1-eq1}, \eqref{local-1-eq1-1}, \eqref{local-1-eq2} and \eqref{local-1-eq2-1}.

\medskip

In fact, it follows from Claim 2 and the fact $X_{1}^{\beta}$ is continuously embedded into $C^{b,\nu}_{\rm unif}(\R^N)$ for $0<\nu\ll 1$ that the mappings
$ t\to u(\cdot,t):=u(t)(\cdot)\in C^{b,\nu}_{\rm unif}(\R^N),\,\,\, t\mapsto v(\cdot,t):=v(t)(\cdot)\in C^{b,\nu}_{\rm unif}(\R^N)$
are locally H\"older continuous in $X$ for $t\in (0,T_{\max})$. By \cite[Lemma 3.3.2]{Dan Henry}, $v(x,t)$  is a classical solution of
 \begin{equation}
\label{v-abstract}
v_{t}=(\Delta -\lambda I) v+\mu u(x,t), \quad x\in\R^N,\,\, 0<t<T_{\max},
\end{equation}
and
$$t\mapsto v_t(\cdot, t)\in C^{b,\nu}_{\rm unif}(\R^N),\,\,\
 t\mapsto \frac{\partial v(\cdot,t)}{\partial x_i}\in C^{b,\nu}_{\rm unif}(\R^N),\,\,\,\,  t\mapsto \frac{\partial^2 v(\cdot,t)}{\partial x_i\partial x_j}\in C^{b,\nu}_{\rm unif}(\R^N)
$$
are also locally H\"older continuous in  $t\in (0,T_{\max})$. Then by the similar arguments to those in
the proof of \cite[Theorem 1.1]{SaSh1}, $(u(x,t),v(x,t))$ is a classical solution of \eqref{Main-Eq} on $(0,T_{\max})$
satisfying \eqref{local-1-eq1}, \eqref{local-1-eq1-1}, \eqref{local-1-eq2} and \eqref{local-1-eq2-1}.
Moreover,  since $u_0\ge 0$ and
$v_{0}\geq 0$, by comparison principle for parabolic equations, we get $u(x,t;u_0,v_0)\geq  0$ and $v(x,t;u_0,v_0)\geq  0$ for all $x\in\R$, $0\le t<T_{\max}$.

\medskip

\noindent {\bf (iii)  Uniqueness.} We now prove that for given $u_0\in C_{\rm unif}^b(\R^N)$, $v_0 \in C_{\rm unif}^{b,1}(\R^N)$ , \eqref{Main-Eq} has a unique classical solution
$(u(\cdot,\cdot;u_0, v_0),v(\cdot,\cdot;u_0, v_0))$ satisfying \eqref{local-1-eq1}, \eqref{local-1-eq1-1}, \eqref{local-1-eq2} and \eqref{local-1-eq2-1}.

\medskip

Any classical solution of \eqref{Main-Eq} satisfying the properties of Theorem \ref{Local-thm1} clearly satisfies the integral equation \eqref{mild-solution-eq}. Suppose that for given $u_0\in C_{\rm unif}^b(\R^N)$, $v_0 \in C_{\rm unif}^{b,1}(\R^N)$  with $u_0\ge 0$,  $v_0\ge 0$, $(u_1(x,t; u_0, v_0),v_1(x,t; u_0, v_0))$ and $(u_2(x,t; u_0, v_0),v_2(x,t; u_0, v_0))$ are two classical solutions of \eqref{Main-Eq} on $\R^N\times[0, T_{\max} )$ satisfying the properties of Theorem \ref{Local-thm1}. Let $0<T<T_{\max}$ be fixed. Thus $\sup_{0\leq t\leq T}(\|u_{1}(\cdot,t; u_0, v_0)\|_{\infty}+\|u_{2}(\cdot,t; u_0, v_0)\|_{\infty})<\infty$ and $\sup_{0\leq t\leq T}(\|v_{1}(\cdot,t; u_0, v_0)\|_{C_{\rm unif}^{b,1}(\R^N)}+\|v_{2}(\cdot,t; u_0, v_0)\|_{C_{\rm unif}^{b,1}(\R^N)}   )<\infty$. Let
$u_i(t)=u_i(\cdot,t; u_0, v_0)$ and $v_{i}(t)=v_i(\cdot,t; u_0, v_0)$ ($i=1,2$).  For every $t\in [0,T]$,  we have that
\begin{align*}%\label{u-equ}
\|u_{1}(t)-u_{2}(t)\|_{\infty}& \leq  \chi C\sup_{0\leq \tau\leq T}(\|\nabla v_1(\tau)\|_{\infty})    \int_{0}^{t}(t-s)^{-\frac{1}{2}}e^{-\lambda(t-s)}\| u_{1}(s)-u_{2}(s)\|_{\infty}ds\nonumber\\
& \,  +\chi C\sup_{0\leq \tau\leq T}(\| u_2(\tau)\|_{\infty}) \int_{0}^{t}(t-s)^{-\frac{1}{2}}e^{-\lambda(t-s)}\|\nabla (v_{2}(s)-v_{1}(s))\|_{\infty}ds \nonumber\\
& \, +(a+\lambda+b\sup_{0\leq \tau\leq T}(\|u_{1}(\tau)\|_{\infty}+\|u_{2}(\tau)\|_{\infty}) ) \int_{0}^{t}e^{-\lambda(t-s)}\|u_{1}(s)-u_{2}(s)\|_{\infty}ds,
\end{align*}
and
\begin{align*}%\label{v-equ}
\|\nabla (v_1(t)-v_2(t))\|_{\infty}&\leq \mu \int_0^t \|\nabla e^{(t-s)(\Delta-\lambda I)}(u_1(s)-u_2(s))\|_{\infty}ds\cr
&\leq \mu C\int_0^t (t-s)^{-\frac{1}{2}} e^{-\lambda (t-s)} \|u_1(s)-u_2(s)\|_{\infty}ds.
\end{align*}
Let $u(t)=u_1(t)-u_2(t)$, $v(t)=v_1(t)-v_2(t)$. We then have
\begin{align}\label{u+vequ}
\|u(t)\|_{\infty}+\|\nabla v(t)\|_{\infty} \leq  M\int_{0}^{t}(t-s)^{-\frac{1}{2}}\left(\|u(s)\|_{\infty}+\|\nabla v(s)\|_{\infty}\right)ds,
\end{align}
where
\begin{align*}
M=& \chi C\sup_{0\leq \tau\leq T}(\|\nabla v_1(\tau)\|_{\infty})+\chi C\sup_{0\leq \tau\leq T}(\| u_2(\tau)\|_{\infty})\\
&\, +   \left(a+\lambda
+ b\sup_{0\leq \tau\leq T}(\|u_{1}(\tau)\|_{\infty}+\|u_{2}(\tau)\|_{\infty})\right) \sqrt{T}
+\mu C <\infty.
\end{align*}
By Lemma \ref{Lem2}, we get $\|u(t)\|_{\infty}\equiv 0$. Thus,
$u_1(t)\equiv u_2(t)$ for all $0\leq t\leq T$.
Since $v(t)=\mu\int_0^t e^{(t-s)(\Delta-\lambda I)}u(s)ds $, then $v(t)\equiv 0$. Hence, $v_1(t)\equiv v_2(t)$ for all $0\leq t\leq T$.
Since $T<T_{\max}$ was arbitrary chosen, then $u_1(t)\equiv u_2(t)$, $v_1(t)\equiv v_2(t)$  for all $0\leq t< T_{\max}$. The theorem is thus proved.
\end{proof}

\subsection{Proof of Theorem \ref{Main-thm1}}

In this subsection, we prove Theorem \ref{Main-thm1}.
\begin{proof}[Proof of Theorem \ref{Main-thm1}]
Assume $b>\frac{N\mu\chi}{4}$.
Theorem \ref{Main-thm1} can be proved by properly modifying the arguments in \cite[Lemma 3.1]{win_JDE2014}.

First, we have
$$
\frac{1}{2}\frac{d}{dt}\left|\nabla v \right|^2=\sum_{i=1}^{N} v_{x_i}(v_{t})_{x_i}.
$$
From the second equation of \eqref{Main-Eq}, we have
\begin{align}\label{nabla-ineq}
\frac{1}{2}\frac{d}{dt}\left|\nabla v \right|^2=\sum_{i=1}^{N} v_{x_i}(\Delta v-\lambda v+\mu u)_{x_i}=\nabla v \cdot \nabla(\Delta v)-\lambda \left|\nabla v\right|^2+\mu \nabla v\cdot \nabla u.
\end{align}
 Note that $\nabla v \cdot \nabla(\Delta v)=\frac{1}{2}\Delta \left|\nabla v \right|^2-\left|D^2 v \right|^2$, \eqref{nabla-ineq} becomes
 \begin{equation}\label{Thm3-1}
 \frac{1}{2\mu}\frac{d}{dt}\left|\nabla v \right|^2 = \frac{1}{2\mu}\Delta \left|\nabla v \right|^2-\frac{1}{\mu}\left|D^2 v \right|^2-\frac{\lambda}{\mu}\left|\nabla v\right|^2+ \nabla v\cdot \nabla u.
 \end{equation}

 Next, multiplying the first equation of \eqref{Main-Eq} by $\frac{1}{\chi}$, we get
 \begin{equation}\label{Thm3-2}
\frac{1}{\chi} u_{t}=\frac{1}{\chi}\Delta u - \nabla u \cdot  \nabla v -u \Delta v+\frac{1}{\chi}u(a-bu).
 \end{equation}
By \eqref{Thm3-1} and \eqref{Thm3-2}, we get
\begin{equation}\label{Thm3-3}
\frac{d}{dt}\big[\frac{1}{\chi} u+ \frac{1}{2\mu}\left|\nabla v \right|^2 \big]=\Delta \big[ \frac{1}{\chi} u+\frac{1}{2\mu}\left|\nabla v \right|^2\big]-\frac{1}{\mu}\left|D^2 v \right|^2-\frac{\lambda}{\mu}\left|\nabla v\right|^2-u \Delta v+\frac{1}{\chi}u(a-bu).
\end{equation}
By Young's inequality, we have
$$
\left|u \Delta v\right|\leq \frac{N\mu}{4}u^2+\frac{1}{\mu}\left|D^2 v \right|^2.
$$
Combining this with \eqref{Thm3-3}, we have
\begin{align}
\frac{d}{dt}\big[\frac{1}{\chi} u+ \frac{1}{2\mu}\left|\nabla v \right|^2 \big]&\leq \Delta \big[ \frac{1}{\chi} u+\frac{1}{2\mu}\left|\nabla v \right|^2\big]-\frac{1}{\mu}\left|D^2 v \right|^2-\frac{\lambda}{\mu}\left|\nabla v\right|^2+\left|u \Delta v\right|+\frac{1}{\chi}u(a-bu)\cr
&\leq  \Delta \big[ \frac{1}{\chi} u+\frac{1}{2\mu}\left|\nabla v \right|^2\big]-\frac{\lambda}{\mu}\left|\nabla v\right|^2+\frac{N\mu}{4}u^2+\frac{1}{\chi}u(a-bu)\cr
&=\Delta \big[ \frac{1}{\chi} u+\frac{1}{2\mu}\left|\nabla v \right|^2\big]-2\lambda\big[\frac{1}{\chi}u+\frac{1}{2\mu}\left|\nabla v\right|^2\big]+\frac{2\lambda}{\chi}u +\frac{N\mu}{4}u^2+\frac{1}{\chi}u(a-bu)\cr
&=\Delta \big[ \frac{1}{\chi} u+\frac{1}{2\mu}\left|\nabla v \right|^2\big]-2\lambda\big[\frac{1}{\chi}u+\frac{1}{2\mu}\left|\nabla v\right|^2\big]
-\frac{1}{\chi}(b-\frac{N\mu\chi}{4} )\big(u-\frac{2(2\lambda+a)}{4b-N\mu\chi  } \big)^2\cr
&\,\,\,\, +\frac{1}{\chi}(b-\frac{N\mu\chi}{4} )\frac{4(2\lambda+a)^2}{(4b-N\mu\chi)^2}.
\end{align}
 Since $b>\frac{N\mu\chi}{4}$ , then for $0<t<T_{max}$, we have
 \begin{equation}
 \frac{d}{dt}\big[\frac{1}{\chi} u+ \frac{1}{2\mu}\left|\nabla v \right|^2 \big] \leq
 \Delta \big[ \frac{1}{\chi} u+\frac{1}{2\mu}\left|\nabla v \right|^2\big] -2\lambda\big[\frac{1}{\chi}u+\frac{1}{2\mu}\left|\nabla v\right|^2\big]
+\frac{(2\lambda+a)^2}{\chi (4b-N\mu\chi)}.
\end{equation}
By the comparison principle for parabolic equations, we have
\begin{equation}\label{u-nablav-bound}
\frac{1}{\chi}u+\frac{1}{2\mu}\left|\nabla v \right|^2\leq \max\{\frac{1}{\chi}\|u_0\|_\infty+\frac{1}{2\mu}\|\nabla v_0\|_\infty^2, \frac{(2\lambda+a)^2}{2\lambda\chi(4b-N\mu\chi)}\} \quad \forall \,\, 0\leq t<T_{max}, x\in\R^{N}.
\end{equation}
Let $M=\max\{\frac{1}{\chi}\|u_0\|_\infty+\frac{1}{2\mu}\|\nabla v_0\|_\infty^2, \frac{(2\lambda+a)^2}{2\lambda\chi(4b-N\mu\chi)}\}$, then
\begin{equation}\label{u-bound}
u(x,t; u_0,v_0)\leq \chi M \quad \forall \,\, 0\leq t<T_{max}, x\in\R^{N},
\end{equation}
and
\begin{equation}\label{nabla-v-bound}
|\nabla v(x,t;u_0,v_0)|\leq \sqrt{2\mu M} \quad \forall \,\, 0\leq t<T_{max}, x\in\R^{N}.
\end{equation}
From the second equation of \eqref{Main-Eq}, by the variation of constant formula,
$$
v(\cdot,t;u_0,v_0)=e^{t(\Delta-\lambda I)}v_0+\mu\int_0^t e^{(t-s)(\Delta-\lambda I)}u(s)ds.
$$
Thus,
\begin{align}\label{v-bound}
\|v(\cdot,t;u_0,v_0)\|_{\infty}&\leq e^{-\lambda t}\|v_0\|_{\infty}+\mu\int_0^t e^{-\lambda(t-s)}\|u(s)\|_{\infty}ds\cr
&\leq e^{-\lambda t}\|v_0\|_{\infty}+\mu\chi M\int_0^t e^{-\lambda(t-s)}ds\cr
&\leq \|v_0\|_{\infty}+\frac{\mu \chi M}{\lambda} \quad \forall \,\, 0\leq t<T_{\max}, x\in\R^{N}.
\end{align}
In view of \eqref{u-bound}, \eqref{nabla-v-bound} and \eqref{v-bound}, we obtain that $\limsup_{t\to T_{\max}}\|u(\cdot,t;u_0, v_0)\|_{\infty}$ is finite and that $\limsup_{t\to T_{\max}}\|v(\cdot,t;u_0,v_0)\|_{C_{\rm unif}^{b,1}(\R^N)}$ is also finite.
Therefore, it follows by the blow-up criterion that $T_{\max}=\infty$ and the solution $(u(x,t; u_0,v_0), v(x,t; u_0,v_0))$ is bounded for $(x,t)\in \R^N\times(0,\infty) $. Thus, \eqref{u-bdd-1} holds by the comparison principle for parabolic equations.

\smallskip

{
Finally, we prove \eqref{u-bdd-2}. To this end, let $U(x,t)=u(x, t; u_0, v_0)-\frac{a}{b}$, $V(x,t)=v(x, t; u_0, v_0)-\frac{\mu}{\lambda}\frac{a}{b}$. Then $(U, V)$ solves
\begin{equation}\label{U-V-eq}
\begin{cases}
U_{t}=\Delta U -\chi \nabla\cdot  ( u(x,t;u_0,v_0)\nabla V) -aU-bU^2,\quad  x\in\R^N,\,\,\, t>0, \\
V_t=\Delta V -\lambda V+\mu U,\quad x\in\R^N,\,\,\, t>0,\\
U(x,0)=u_0(x)-\frac{a}{b},\quad x\in\R^N, \\
V(x,0)=v_0(x)-\frac{\mu}{\lambda}\frac{a}{b}, \quad x\in\R^N.
\end{cases}
\end{equation}
%We first present some lemmas on the estimates of $U$ and $V$. The first lemma provides some estimates for
%$\|U(\cdot,t)+ \frac{\chi}{2\mu}\left|\nabla V(\cdot,t) \right|^2\|_\infty$ and $\|u(\cdot,t;u_0,v_0)\|$.
%There hold
%\begin{equation}
%\label{U-V-estimate}
%\limsup_{t\to\infty}\|U(\cdot,t)+ \frac{\chi}{2\mu}\left|\nabla V(\cdot,t) \right|^2\|_\infty \le \frac{aN\mu\chi}{b(4b-N\mu\chi)}\le \frac{a}{b(1-\theta)}
%\end{equation}{\color{blue}\textbf{(I think we can not get \eqref{U-V-estimate}, because $U(\cdot,t)$ may be negative. \eqref{U-V-estimate} should be replaced by
%$$
%\limsup_{t\to\infty}\|U_{+}(\cdot, t)\|_{\infty}\leq \frac{ a\theta}{b(1-\theta)}.
%$$
%)}}
%and
%\begin{equation}\label{u-estimate}
%\limsup_{t\to\infty}\|u(\cdot, t; u_0,v_0)\|_{\infty}\leq \frac{a}{b(1-\theta)}.
%\end{equation}

By \eqref{U-V-eq} and the similar arguments as those used in deriving \eqref{Thm3-3}, we have
\begin{equation}\label{New-U-equ}
\frac{d}{dt}\big[U+ \frac{\chi}{2\mu}\left|\nabla V \right|^2\big]=
\Delta \big[ U+\frac{\chi}{2\mu}\left|\nabla V \right|^2\big]-\chi u \Delta V-aU-bU^2
-\frac{\chi}{\mu}\left|D^2 V \right|^2-\frac{\chi\lambda}{\mu}\left|\nabla V\right|^2.
\end{equation}
It follows from Young's inequality and $U=u-\frac{a}{b}$ that
\begin{align}\label{New-U-V-ineq}
\left|u \Delta V\right|\leq \frac{N\mu}{4}u^2+\frac{1}{\mu}\left|D^2 V \right|^2=\frac{N \mu}{4}U^2+\frac{N\mu a}{2b}U+\frac{a^2N\mu}{4b^2}+\frac{1}{\mu}\left|D^2 V \right|^2.
\end{align}
By \eqref{New-U-V-ineq}, $b>\frac{N\chi \mu}{4}$, and $\lambda\geq\frac{a}{2}$, we have
\begin{align}
\frac{d}{dt}\big[U+ \frac{\chi}{2\mu}\left|\nabla V \right|^2 \big]&\le \Delta \big[ U+\frac{\chi}{2\mu}\left|\nabla V \right|^2\big]-a\big( U+\frac{\chi}{2\mu}\left|\nabla V \right|^2\big)+\frac{a^2N\mu\chi}{b(4b-N\mu\chi)}.
\end{align}
This together with  the comparison principle for parabolic equations and $4b>N\chi\mu$ implies
\begin{equation}\label{U-V-estimate}
\limsup_{t\to\infty}\sup_{x\in\R^N}\big(U(\cdot,t)+ \frac{\chi}{2\mu}\left|\nabla V(\cdot,t) \right|^2\big) \le \frac{aN\mu\chi}{b(4b-N\mu\chi)}.
\end{equation}
We then have
\begin{equation}\label{U-posi-bdd}
\limsup_{t\to\infty}\|U_{+}(\cdot, t)\|_{\infty}\leq \frac{ aN\mu\chi}{b(4b-N\mu\chi)}.
\end{equation}
and
%\begin{equation}
%\label{U-V-estimate}
%\limsup_{t\to\infty}\|U(\cdot,t)+ \frac{\chi}{2\mu}\left|\nabla V(\cdot,t) \right|^2\|_\infty \le \frac{aN\mu\chi}{b(4b-N\mu\chi)}\le \frac{a}{b(1-\theta)}
%\end{equation}{\color{blue}\textbf{(I think we can not get \eqref{U-V-estimate}, because $U(\cdot,t)$ may be negative. \eqref{U-V-estimate} should be replaced by
%$$
%\limsup_{t\to\infty}\|U_{+}(\cdot, t)\|_{\infty}\leq \frac{ a\theta}{b(1-\theta)}.
%$$
\begin{align*}
\limsup_{t\to\infty}\|u(\cdot, t; u_0, v_0)\|_{\infty}\leq \frac{a}{b}+\frac{aN\mu\chi}{b(4b-N\mu\chi)}=\frac{4a}{4b-N\mu\chi}.
\end{align*}
Hence \eqref{u-bdd-2} holds and the theorem is thus proved.
}
\end{proof}

\section{Persistence}

In this section, we investigate the persistence of  global classical solutions of \eqref{Main-Eq} with strictly positive initial functions $u_0$ and prove Theorem \ref{Main-thm3}.
Throughout this section, unless specified otherwise,
$C$ denotes a generic  constant
independent of $u,v$ and may be different at different places. Recall that $X_1=C_{\rm unif}^b(\R^N)$, $X_2= C_{\rm unif}^{b,1}(\R^N)$. Let
$$
X_1^+=\{u\in X_1\,|\, u\ge 0\},\quad X_2^+=\{v\in X_2\,|\, v\ge 0\}.
$$
We first prove a lemma.

\begin{lem}
\label{persistence-lm}
Suppose that $b>\frac{N\chi\mu}{4}$. There are  $M>0$, $M_1>0$, and $0<\theta<\frac{1}{2}$ such that for any $(u_0,v_0)\in X_1^+\times X_2^+$, there is
$T_0(u_0,v_0)>1$ satisfying that
\begin{equation}
\label{condition-on-initial-eq}
\begin{cases}
\|u(\cdot,t;u_0,v_0)\|_\infty\le M\quad \forall \, t\ge T_0(u_0,v_0)\cr
 \|v(\cdot,t;u_0,v_0)\|_\infty\le M\quad \forall\, t\ge T_0(u_0,v_0)\cr
\|\nabla v(\cdot,t;u_0,v_0)\|_\infty\le M\quad \forall\, t\ge T_0(u_0,v_0)\cr
 \|\Delta v(\cdot,t;u_0,v_0)\|_\infty\le M\quad \forall\, t\ge T_0(u_0,v_0)
 \end{cases}
 \end{equation}
 and
 \begin{equation}
 \label{condition-on-initial-eq-0}
 \sup_{t,s\ge T_0(u_0,v_0)+1, t\not =s} \frac{\|\nabla v(\cdot,t;u_0,v_0)-\nabla v(\cdot,s;u_0,v_0)\|_\infty}{|t-s|^\theta}
 \le M M_1.
\end{equation}
 \end{lem}

\begin{proof}
First, we obtain the upper bound for $\|u(\cdot,t;u_0,v_0)\|_\infty$ and
$\|\nabla v(\cdot,t;u_0,v_0)\|_\infty$.
By \eqref{u-nablav-bound}, it holds that
$$
\limsup_{t\to\infty}\|u(\cdot,t; u_0, v_0)\|_{\infty}\leq \frac{(2\lambda+a)^2}{2\lambda(4b-N\mu\chi)}
$$
and
$$
\limsup_{t\to\infty}\|\nabla v(\cdot,t; u_0, v_0)\|_{\infty}\leq \sqrt{\frac{\mu (2\lambda+a)^2}{\lambda \chi (4b-N\mu\chi)}}.
$$
It thus follows that there exists $T_1=T_1(u_0,v_0)$ such that
\begin{equation}\label{u-bdd}
\|u(\cdot,t; u_0,v_0)\|_{\infty}\leq \frac{(2\lambda+a)^2}{\lambda(4b-N\mu\chi)}  \quad\,\ \forall\,\ t\geq T_1
\end{equation}
and
\begin{equation}\label{nablav-bdd}
\|\nabla v(\cdot,t;u_0,v_0)\|_{\infty}\leq 2\sqrt{\frac{\mu (2\lambda+a)^2}{\lambda \chi (4b-N\mu\chi)}} \quad\,\ \forall\,\ t\geq T_1.
\end{equation}

Next, we obtain the upper bound for $\|v(\cdot,t;u_0,v_0)\|_\infty$.
By the variation of constant formula, we have that
$$
v(\cdot,t; u_0, v_0)=e^{(t-T_1)(\Delta-\lambda I)}v(\cdot, T_1; u_0, v_0)+\mu\int_{T_1}^t e^{(t-s)(\Delta-\lambda I)}u(\cdot, s; u_0, v_0)ds \quad \forall\,\ t\geq T_1.
$$
Then by \eqref{Lp-Lq-1} and \eqref{u-bdd},
\begin{align*}
\|v(\cdot,t; u_0, v_0)\|_{\infty}&\leq e^{-\lambda (t-T_1)}\|v(\cdot, T_1; u_0, v_0)\|_{\infty}+\mu \int_{T_1}^t e^{-\lambda(t-s)}\|u(\cdot, s; u_0, v_0)\|_{\infty}ds\cr
&\leq  e^{-\lambda (t-T_1)}\|v(\cdot, T_1; u_0, v_0)\|_{\infty}+\frac{\mu(2\lambda+a)^2}{\lambda^2(4b-N\mu\chi)}\quad \forall\,\ t\geq T_1.
\end{align*}
This implies that
$$
\limsup_{t\to\infty}\|v(\cdot,t; u_0, v_0)\|_{\infty}\leq \frac{\mu(2\lambda+a)^2}{\lambda^2(4b-N\mu\chi)}.
$$
It thus follows that there exists $T_2=T_2(u_0,v_0)>T_1$ such that
\begin{equation}\label{v-bdd}
\|v(\cdot,t; u_0,v_0)\|_{\infty}\leq \frac{2\mu(2\lambda+a)^2}{\lambda^2(4b-N\mu\chi)}  \quad\,\ \forall\,\ t\geq T_2.
\end{equation}

Now, we obtain the upper bound for
$\|\Delta v(\cdot,t;u_0,v_0)\|_\infty$. To this end,
first,  we write $u(x,t)$ and $v(x,t)$ for $u(x,t;u_0,v_0)$ and $v(x,t;u_0,v_0)$.
By the variation of constant formula again, we have
\begin{align*}
u(\cdot,t)=&\underbrace{e^{(t-T_1)(\Delta-\lambda I)}u(\cdot, T_1)}_{I_1}-\underbrace{\chi\int_{T_1}^t e^{(t-s)(\Delta-\lambda I)}\nabla\cdot(u(\cdot, s)\nabla v(\cdot, s))ds}_{I_2}\\
&\,\, +\underbrace{\int_{T_1}^t e^{(t-s)(\Delta-\lambda I)}u(\cdot, s)(a+\lambda-bu(\cdot, s))ds}_{I_3}
\end{align*}
for $t\ge T_1$.
Let $A^{\beta}=(\lambda I-\Delta)^{\beta}$ for $0<\beta<\frac{1}{2}$. By \eqref{Lp-Lq-3}, Lemma \ref{L_Infty bound 2}, \eqref{u-bdd} and \eqref{nablav-bdd}, we have
\begin{equation}\label{u-01-01}
\|A^\beta I_1\|_{\infty}=\|A^{\beta} e^{(t-T_1)(\Delta-\lambda I)}u(\cdot, T_1)\|_{\infty}\leq C(t-T_1)^{-\beta}e^{-\lambda (t-T_1)}\|u(\cdot, T_1)\|_{\infty} \quad \forall\,\ t\geq T_1,
\end{equation}
\begin{align}\label{u-02-02}
\|A^\beta I_2\|_\infty&\le \chi\int_{T_1}^{t}{\|A^{\beta}e^{(t-s)(\Delta-\lambda I)}\nabla\cdot (u(\cdot, s)\nabla v(\cdot, s))ds} \|_{\infty}ds \cr
&\leq  \chi C\int_{T_1}^{t}(t-s)^{-\beta-\frac{1}{2}}e^{-\lambda(t-s)}\|u(\cdot, s)\|_{\infty}\|\nabla v(\cdot, s)\|_{\infty}ds\nonumber\\
%&\leq  \chi C\int_{t_2}^{t}(t-s)^{-\beta-\frac{1}{2}}e^{-\lambda(t-s)}ds\nonumber\\
&\leq \chi C \lambda^{\beta-\frac{1}{2}} \frac{(2\lambda+a)^2}{\lambda(4b-N\mu\chi)} \sqrt{\frac{\mu (2\lambda+a)^2}{\lambda \chi (4b-N\mu\chi)}}  \quad \forall\,\ t\geq T_1,
\end{align}
and
\begin{align}\label{u-03-03}
\|A^\beta I_3\|_\infty &\le \int_{T_1}^t \|A^{\beta}e^{(t-s)(\Delta-\lambda I)}u(\cdot, s)(a+\lambda-bu(\cdot, s))\|_{\infty}ds\cr
&\leq C \int_{T_1}^{t}(t-s)^{-\beta}e^{-\lambda(t-s)}\big((a+\lambda)\|u(\cdot,s)\|_{\infty}+b\|u(\cdot,s)\|_{\infty}^{2}\big)ds\cr
%&\leq C\big((a+\lambda)\chi M+b\chi^2M^2 \big)\int_{0}^{t}(t-s)^{-\beta}e^{-\lambda(t-s)} ds\cr
&\leq C\big((a+\lambda)\frac{(2\lambda+a)^2}{\lambda(4b-N\mu\chi)}+b\frac{(2\lambda+a)^4}{\lambda^2(4b-N\mu\chi)^2} \big)\lambda^{\beta-1} \quad \forall\,\ t\geq T_1.
\end{align}
By \eqref{u-01-01}, \eqref{u-02-02}, and \eqref{u-03-03},  there exists $T_3:=T_3(u_0,v_0)\geq T_1$ such that
\begin{align}\label{A-beta-u-2}
\|A^{\beta}u(\cdot,t)\|_{\infty}&\leq %C(t-T_1)^{-\beta}e^{-\lambda (t-T_1)}\|u(T_1,\cdot)\|_{\infty}\cr
2\chi C \lambda^{\beta-\frac{1}{2}} \frac{(2\lambda+a)^2}{\lambda(4b-N\mu\chi)} \sqrt{\frac{\mu (2\lambda+a)^2}{\lambda \chi (4b-N\mu\chi)}} \cr
&+2C\big((a+\lambda)\frac{(2\lambda+a)^2}{\lambda(4b-N\mu\chi)}+b\frac{(2\lambda+a)^4}{\lambda^2(4b-N\mu\chi)^2} \big)\lambda^{\beta-1} \quad \forall\,\ t\geq T_3.
\end{align}

Fix $\gamma\in(1,\frac{3}{2})$, and then choose $\beta$ such that $\gamma-1<\beta<\frac{1}{2}$. Note that
there is a constant $C$ independent of $u_0,v_0$ and $t$ such that
\begin{equation}\label{C2-embeding-1}
\|v(\cdot,t;u_0,v_0)\|_{C^2}\leq C \|v(\cdot,t;u_0,v_0)\|_{X^{\gamma}}.
\end{equation}
By the variation of constant formula again,
$$
v(\cdot,t;u_0,v_0)=\underbrace{e^{(t-T_3)(\Delta-\lambda I)}v(\cdot, T_3;u_0,v_0)}_{J_1}+
\underbrace{\mu\int_{T_3}^{t}e^{(t-s)(\Delta-\lambda I)}u(\cdot, s;u_0,v_0)ds}_{J_2} \quad \forall\,\ t\geq T_3.
$$
Note that
\begin{align}\label{A-gamma-u-1}
\|A^\gamma J_1\|_\infty&=\|A^{\gamma}e^{(t-T_3)(\Delta-\lambda I)}v(\cdot, T_3,;u_0,v_0)\|_{\infty}\cr
&\leq C(t-T_3)^{-\gamma}e^{-\lambda(t-T_3)}\|v(\cdot, T_3;u_0,v_0)\|_{\infty},
%&\leq 2C_{\gamma}(t- \tilde T_0)^{-\gamma}e^{-\lambda(t-\tildet_0)}\|v_0\|_{\infty}
\end{align}
%Moreover, by \eqref{Lp-Lq-3} and \eqref{A-beta-bd}, we have
and
\begin{align}\label{A-gamma-u-2}
\|A^\gamma J_2\|_\infty&\le \mu\int_{T_3}^{t}\|A^{\gamma}e^{(t-s)(\Delta-\lambda I)}u(\cdot, s;u_0,v_0)\|_{\infty}ds\cr
&=\mu\int_{T_3}^{t}\|A^{\gamma-\beta}e^{(t-s)(\Delta-\lambda I)}A^{\beta}u(\cdot, s;u_0,v_0)\|_{\infty}ds\cr
&\leq \mu C\int_{T_3}^{t}(t-s)^{-(\gamma-\beta)}e^{-\lambda(t-s)}\|A^{\beta}u(\cdot, s;u_0,v_0)\|_{\infty}ds\cr
%&\leq \mu C\sup_{t\ge T_3}\|A^{\beta}u(\cdot,t;u_0,v_0)\|_{\infty}\int_{T_3}^{t}(t-s)^{-(\gamma-\beta)}e^{-\lambda(t-s)}ds\cr
&\leq \mu C\lambda^{\gamma-\beta-1}\sup_{t\ge T_3}\|A^{\beta}u(\cdot,t;u_0,v_0)\|_{\infty}.
\end{align}
It follows from \eqref{A-beta-u-2}, \eqref{C2-embeding-1}, \eqref{A-gamma-u-1}, and \eqref{A-gamma-u-2}  that
\begin{align}\label{delta-v}
\limsup_{t\to\infty}\|v(\cdot,t;u_0,v_0)\|_{C^2}&\leq
\mu C\lambda^{\gamma-\beta-1}
\bigg(\chi  \lambda^{\beta-\frac{1}{2}} \frac{(2\lambda+a)^2}{\lambda(4b-N\mu\chi)} \sqrt{\frac{\mu (2\lambda+a)^2}{\lambda \chi (4b-N\mu\chi)}} \cr
&\,\,\,+\big((a+\lambda)\frac{(2\lambda+a)^2}{\lambda(4b-N\mu\chi)}+b\frac{(2\lambda+a)^4}{\lambda^2(4b-N\mu\chi)^2} \big)\lambda^{\beta-1}\bigg).
\end{align}
By \eqref{u-bdd}, \eqref{nablav-bdd}, \eqref{v-bdd}, and \eqref{delta-v}, there are $M>0$ and
$T_0(u_0,v_0)>1$ such that \eqref{condition-on-initial-eq} holds.

Finally, by the arguments in Claim 2 of Theorem \ref{Local-thm1}, there are
$0<\theta<\frac{1}{2}$ and $M_1>1$ such that for any $(u_0,v_0)\in X_1^+\times X_2^+$ satisfying \eqref{condition-on-initial-eq},
\eqref{condition-on-initial-eq-0} holds.
\end{proof}

Throughout the rest of this section, $M$ and $M_1$ are as in Lemma \ref{persistence-lm}.  $L_0\ge  1$  is a fixed positive number  such that $\lambda_0>0$, where
$\lambda_0$ is the principal eigenvalue of
\begin{equation}
\begin{cases}\label{eigen-eq}
\Delta \phi +\frac{a}{2}\phi=\lambda \phi, \quad x\in B_{L_0}(0),\\
\phi(x)=0, \quad x\in\partial B_{L_0}(0).
\end{cases}
\end{equation}

We now prove  Theorem \ref{Main-thm3}.

\begin{proof}[Proof of Theorem \ref{Main-thm3}]
We divide the proof into four steps (see the main idea of the proof in Remark 3 in the introduction).

\medskip

\noindent {\bf Step 1.} %For given $R>0$ and $x_0\in\R^N$, let $B_{R}(x_0)=\{x\in\R^N\,|\, |x-x_0|<R\}$.
 In this step, we  prove that {\it there is $\epsilon_0>0$ such that for any $0<\epsilon\le \epsilon_0$,  any
$(u_0,v_0)\in X_1^+\times X_2^+$, any $x_0\in \R^N$, and any $t_1,t_2$ satisfying $T_0(u_0,v_0)\le t_1< t_2\le  \infty$, if
$$
\sup_{x\in{B}_{2L(\epsilon)}(x_0)} u(x,t;u_0,v_0)\le \epsilon\quad \forall\, t_1\le t< t_2,
$$
then
\begin{equation}
\label{v-dv-est}
\sup_{x\in{B}_{L(\epsilon)}(x_0)}\max\{ v(x,t;u_0,v_0),|\p_{x_i} v(x,t;u_0,v_0)|\}\leq \tilde M\epsilon\quad \forall\, t_1+T(\epsilon)\le t< t_2,\,\, i=1,2,\cdots,N,
\end{equation}
 and
 \begin{equation}
 \label{ddv-dtv-est}
 \chi \sup_{x\in B_{L(\epsilon)}(x_0)} \sum_{i,j=1}^N|\p_{x_ix_j} v(x,t;u_0,v_0)|\le \frac{a}{4}\quad \forall\, t_1+T(\epsilon)+1\le t<t_2,
 \end{equation}
 where
 $$
 \tilde M=\max\big \{  1+\frac{\mu M}{\lambda \pi^{\frac{N}{2}}}+\frac{\mu}{\lambda},\,\,  1+\frac{ \mu}{\pi^{\frac{N}{2}}}\lambda^{-\frac{1}{2}}\Gamma(\frac{1}{2})M+\frac{ \mu}{\pi^{\frac{N}{2}}}\lambda^{-\frac{1}{2}}\Gamma(\frac{1}{2})\big \};
 $$
 $T=T(\epsilon)\ge 1$ is such that
\begin{equation}\label{T-choice}
 e^{-\lambda T}M\leq \epsilon;
\end{equation}
 and $L=L(\epsilon)\geq L_0$ is such that
\begin{equation}\label{L-choice}
\max\{\int_{\R^{N}\backslash {B}_{\frac{L}{2\sqrt{2T}}}(0)}e^{-|z|^2}dz, \int_{\R^{N}\backslash {B}_{\frac{L}{2\sqrt{2T}}}(0)}|z|e^{-|z|^2}dz\}\leq \epsilon.
\end{equation} }

\medskip

We first prove that \eqref{v-dv-est} holds for any $\epsilon>0$.
Fix $t_1\geq T_0(u_0,v_0)$. Note that
 \begin{align*}
v(x,t;u_0,v_0)&=\int_{\R^N}\frac{e^{-\lambda (t-t_1)}}{(4\pi (t-t_1))^{\frac{N}{2}}}e^{-\frac{|x-y|^{2}}{4(t-t_1)}}v(y, t_1;u_0,v_0)dy   \cr
&+\mu \int_{t_1}^{t}\int_{\R^N}\frac{e^{-\lambda (t-s)}}{(4\pi (t-s))^{\frac{N}{2}}}e^{-\frac{|x-y|^{2}}{4(t-s)}}u(y, s;u_0,v_0)dyds\cr
&=\frac{1}{\pi^{\frac{N}{2}}}\int_{\R^N}e^{-\lambda (t-t_1)}e^{-|z|^2}v(x+2\sqrt{t-t_1}z, t_1;u_0,v_0)dz   \cr
&+\frac{\mu}{\pi^{\frac{N}{2}}} \int_{t_1}^{t}\int_{\R^N}e^{-\lambda (t-s)}e^{-|z|^2}u(x+2\sqrt{t-s}z, s;u_0,v_0)dzds,
\end{align*}
and
\begin{align}\label{v-deriv-est}
\partial_{x_i}v(x,t;u_0,v_0)&=\int_{\R^N}\frac{(y_i-x_i)e^{-\lambda (t-t_1)}}{2(t-t_1)(4\pi (t-t_1))^{\frac{N}{2}}}e^{-\frac{|x-y|^{2}}{4(t-t_1)}}v(y, t_1;u_0,v_0)dy \cr
&+\mu \int_{t_1}^{t}\int_{\R^N}\frac{(y_i-x_i)e^{-\lambda (t-s)}}{2(t-s)(4\pi (t-s))^{\frac{N}{2}}}e^{-\frac{|x-y|^{2}}{4(t-s)}}u(y, s;u_0,v_0)dyds\cr
&=\frac{1}{\pi^{\frac{N}{2}}}(t-t_1)^{-\frac{1}{2}}e^{-\lambda (t-t_1)}\int_{\R^N}ze^{-z^2}v(x+2\sqrt{t-t_1}z, t_1;u_0,v_0)dz\cr
&+\frac{\mu}{\pi^{\frac{N}{2}}} \int_{t_1}^{t}\int_{\R^N}(t-s)^{-\frac{1}{2}}e^{-\lambda (t-s)}ze^{-z^2}u(x+2\sqrt{t-s}z, s;u_0,v_0)dzds.
\end{align}
Hence, for any $x_0\in\R^N$, $x\in {B}_{L}(x_0)$, and $t_1+T\leq t\leq \min\{t_1+2T,t_2\}$, we have
\begin{align*}
v(x,t;u_0,v_0)&\leq e^{-\lambda T}M+\frac{ \mu}{\pi^{\frac{N}{2}}} \left[\int_{t_1}^{t}\int_{\R^N\backslash {B}_{\frac{L}{2\sqrt{2T}}}(0)}e^{-\lambda (t-s)}e^{-|z|^2}dzds\right]M \cr
&+\frac{ \mu}{\pi^{\frac{N}{2}}} \left[\int_{t_1}^{t}\int_{{B}_{ \frac{L}{2\sqrt{2T}}}(0)}e^{-\lambda (t-s)}e^{-|z|^2}dzds\right]\sup_{t_1\leq t< t_2, z\in {B}_{2L}(x_0)}u(z, t;u_0,v_0).
%& +\frac{ \mu}{\pi^{\frac{N}{2}}} \left[\int_{t_1}^{t}\int_{\R\backslash {B}_{\frac{L}{2\sqrt{2T}}}}e^{-\lambda (t-s)}e^{-|z|^2}dzds\right]M
\end{align*}
By \eqref{T-choice} and \eqref{L-choice},  if $\sup_{x\in{B}_{2L}(x_0)} u(x,t;u_0,v_0)\le \epsilon$ for any $t_1\le t< t_2$, then
\begin{equation}\label{v-est}
v(x,t;u_0,v_0)\leq (1+\frac{\mu M}{\lambda \pi^{\frac{N}{2}}}+\frac{\mu}{\lambda})\epsilon \quad \forall\,\ t_1+T\leq t\leq \min\{t_1+2T,t_2\},\,\
x\in {B}_{L}(x_0).
\end{equation}
For $t_1+T\leq t\leq \min\{t_1+2T,t_2\}$, and  $x\in {B}_{L}(x_0)$, we have
\begin{align}\label{partial-v-est}
&|\partial_{x_i}v(x,t;u_0,v_0)|\cr
& \leq \frac{1}{\pi^{\frac{N}{2}}}T^{-\frac{1}{2}}e^{-\lambda T}M+\frac{ \mu}{\pi^{\frac{N}{2}}} \left[\int_{t_1}^{t}\int_{\R^N\backslash {B}_{\frac{L}{2\sqrt{2T}}}(0)}(t-s)^{-\frac{1}{2}}e^{-\lambda (t-s)}|z|e^{-|z|^2}dzds\right]M \cr
&+\frac{ \mu}{\pi^{\frac{N}{2}}} \left[\int_{t_1}^{t}\int_{{B}_{ \frac{L}{2\sqrt{2T}}}(0)}(t-s)^{-\frac{1}{2}}e^{-\lambda (t-s)}|z|e^{-|z|^2}dzds\right]\sup_{t_1\leq t< t_2, z\in {B}_{2L}(x_0)}u(z,t;u_0,v_0).
\end{align}
By \eqref{T-choice} and \eqref{L-choice},  if $\sup_{x\in{B}_{2L}(x_0)} u(x,t;u_0,v_0)\le \epsilon$ for any $t_1\le t< t_2$, then
\begin{equation}
\label{partial-v-est-1}
|\partial_{x_i}v(x,t;u_0,v_0)|\leq (1+\frac{ \mu}{\pi^{\frac{N}{2}}}\lambda^{-\frac{1}{2}}\Gamma(\frac{1}{2})M+\frac{ \mu}{\pi^{\frac{N}{2}}}\lambda^{-\frac{1}{2}}\Gamma(\frac{1}{2}))\epsilon
 \end{equation}
 for  $t_1+T\leq t\le \min\{t_1+2T, t_2\}$ and $x\in {B}_{L}(x_0)$.

 In the above arguments, replace $t_1$ by $t_1+T$. We have \eqref{v-est} and \eqref{partial-v-est-1}
 for $t_1+2T\le t\le \min\{t_1+3T,t_2\}$. Repeating this process, we have \eqref{v-est} and \eqref{partial-v-est-1}
 for $t_1+T\le t<t_2$.
It then follows that \eqref{v-dv-est} holds for any $\epsilon>0$.

Next, we prove that there is $\epsilon_0>0$ such that \eqref{ddv-dtv-est} holds for $0<\epsilon\le \epsilon_0$. Assume this is  not true. Then there are $\epsilon_n\to 0$,
$(u_n,v_n)\in X_1^+\times X_2^+$, $x_n\in\R^N$, $T_0(u_n,v_n)\le t_{1n}<t_{1n}+T(\epsilon_n)+1\le t_n<t_{2n}$ such that
$$
\sup_{x\in B_{2L(\epsilon_n)}(x_n)}u(x,t;u_n,v_n)\le \epsilon_n, \quad \forall\, t_{1n}\le t< t_{2n}
$$
and
$$
\chi \sup_{x\in B_{L(\epsilon_n)}(x_n)}\sum_{i,j=1}^N |\p_{x_ix_j}v(x,t_n;u_n,v_n)|> \frac{a}{4}.
$$
Let
$$
(u_n(x,t),v_n(x,t))=(u(x+x_n,t+t_n;u_n,v_n),v(x+x_n,t+t_n;u_n,v_n)).
$$
Without loss of generality, we may assume that
$$
(u_n(x,t),v_n(x,t))\to (u^*(x,t),v^*(x,t))
$$
as $n\to \infty$ locally uniformly on $(x,t)\in\R^N\times [-1,\infty)$. Note that $v^*(x,t)$ satisfies
$$
v_t^*=\Delta v^*-\lambda v^*+\mu u^*,\quad x\in\R^N,\,\, t\ge -1
$$
and
$$
\chi \sup_{{x\in \R^{N}}}\sum_{i,j=1}^N|\p_{x_ix_j}v^*(x,0)|\geq \frac{a}{4}.
$$
By \eqref{v-dv-est}, we have
$$
u^*(x,t)=0,\,\, v^*(x,t)=0\quad {x\in \R^{N}},\,\, -1\le t \le 0.
$$
Then by the comparison principle for parabolic equations,
$$
v^*(x,t)= 0\quad\forall\, x\in\R^N,\,\, t\ge -1,
$$
which is a contradiction. Hence  \eqref{ddv-dtv-est} holds.

\medskip

\noindent {\bf Step 2.}  Let $\tilde T_0\geq 1$ be such that
$e^{\lambda_0\tilde T_0}\ge 3$.
 In this step, we prove that
{\it for any $0<\epsilon\le \epsilon_0$, there is $0<\delta_{\epsilon}\leq \epsilon_0$ such that for any $(u_0,v_0)\in X_1^+\times X_2^+$, $x_0\in\R^N$, any  $t_0\ge T_0(u_0,v_0)+2$,  if
$$
\sup_{x\in{B}_{2L}(x_0)} u(x, t_0;u_0,v_0)\ge \epsilon,
$$
then
$$
\inf_{x\in{B}_{2L}(x_0)} u(x,t;u_0,v_0)\ge \delta_{\epsilon}\quad \forall\, t_0\le t\le t_0+T+\tilde T_0,
$$
where $L=L(\epsilon)$ and $T=T(\epsilon)$ are as in Step 1.
}

\medskip

{
Suppose on the contrary that the conclusion in Step 2 fails. Then there exist $0<\bar\epsilon_0\le \epsilon_0$,
$(u_{0n},v_{0n})\in X_1^+\times X_2^+$, $x_{0n}\in\R^N$, $t_{0n}\geq T_0(u_{0n},v_{0n})+2$, $x_n$, $x^*_{n}\in\R^N$ with $|x_n-x_{0n}|\leq 2L(\bar\epsilon_0)$, $|x^*_{n}-x_{0n}|\leq 2L(\bar\epsilon_0)$, $t_{n}\in\R$ with $t_{0n}\leq t_{n}\leq t_{0n}+T(\bar\epsilon_0)+ \tilde T_0$ such that
\begin{equation}\label{u-t-0-n}
\lim_{n\to\infty} u(x_n, t_{0n}; u_{0n}, v_{0n})\geq \bar\epsilon_0
\end{equation}
and
\begin{equation}\label{u-t-n}
\lim_{n\to\infty} u(x^*_{n}, t_{n}; u_{0n}, v_{0n})=0.
\end{equation}
Let $u_n(x,t)=u(x+x_{0n} ,t+t_{0n}-1;u_{0n},v_{0n})$, $v_n(x,t)=v(x+x_{0n}, t+t_{0n}-1;u_{0n},v_{0n})$, and
$T=T(\bar\epsilon_0)+ \tilde T_0$, $L=L(\bar\epsilon_0)$.
Without loss of generality,
we may assume that
\begin{equation}\label{u-n-conv}
(u_n(x,t),v_n(x,t))\to (u^*(x,t),v^*(x,t))
\end{equation}
as $n\to\infty$ locally uniformly in $(x,t)\in \R^N\times [0,\infty)$. Then  $(u^*, v^*)$ is a solution of \eqref{Main-Eq}
for $t\ge 0$.

By \eqref{u-t-0-n}, \eqref{u-n-conv} and the comparison principle for parabolic equations, we have
$$
\inf_{t\in[1,1+T], x\in B_{2L}(0)} u^*(x,t)>0.
$$
Let $\tilde\epsilon=\frac{1}{2}\inf_{t\in[1,1+T], x\in B_{2L}(0)} u^*(x,t)>0$. By \eqref{u-n-conv}, there exists $N=N(\tilde \epsilon)$ such that
\begin{align*}
u_n(x,t) \geq u^*(x,t)- \tilde\epsilon \geq \inf_{t\in[1,1+T], x\in B_{2L}(0)} u^*(x,t)-\tilde \epsilon=\frac{1}{2}\inf_{t\in[1,1+T], x\in B_{2L}(0)} u^*(x,t)=\tilde\epsilon
\end{align*}
for any $t\in [1,1+T], x\in B_{2L}(0)$ and $n\ge N$.
Thus,
$$\inf_{t\in[1,1+T], x\in B_{2L}(0)} u_n(x,t)\geq \tilde\epsilon\quad \forall\,\, n\ge N.$$
Note that $$\inf_{t\in[1,1+T], x\in B_{2L}(0)} u_n(x,t)=\inf_{t\in[t_{0n},t_{0n}+T], x\in B_{2L}(x_{0n})} u(x,t;u_{0n}, v_{0n}).$$
It follows that
$$\liminf_{n\to\infty}\inf_{t\in[t_{0n},t_{0n}+T], x\in B_{2L}(x_{0n})} u(x,t;u_{0n}, v_{0n})\geq \tilde\epsilon,$$
which contradicts to \eqref{u-t-n}. Step 2 is thus proved.
}

\medskip

\noindent {\bf Step 3.} In this step, we prove that {\it  there is $ 0<\tilde\epsilon_0\le \epsilon_0$ such that for any $0<\epsilon\le  \tilde\epsilon_0$, there is $\tilde\delta_\epsilon>0$ such that
for any $(u_0,v_0)\in X_1^+\times X_2^+$, any $x_0\in\R^N$, and any $t_1,t_2$ satisfying that $T_0(u_0,v_0){+2}\le t_1< t_2\le \infty$,
if
$$
\sup_{x\in{B}_{2L}(x_0)}u(x, t_1;u_0,v_0)=\epsilon,\, \,\, \sup_{x\in{B}_{2L}(x_0)} u(x,t;u_0,v_0)\le \epsilon,\,\, \forall\,\ t_1<t<t_2,
$$
then
$$
\inf_{x\in{B}_{2L}(x_0)} u(x,t;u_0,v_0)\ge \tilde\delta_\epsilon\quad \forall\,\ t_1\le t< t_2,
$$
where $L=L(\epsilon)$ is as in Step 1.
}

{To prove the statement in Step 3, first, let $\phi_0$ be the positive principal eigenfunction of \eqref{eigen-eq}
with $\sup_{x\in B_{L_0}(0)}\phi_0(x)=1$.
Consider
\begin{equation}
\label{new-new-eqq1}
\begin{cases}
 u_t=\Delta  u+q(x,t)\cdot\nabla  u+\frac{a}{2}u, \quad x\in B_{L_0}(0), t>0\\
 u(x,t)=0, \quad x\in\partial B_{L_0}(0), \quad t>0,\\
 u(x,0)=\phi_0(x), \quad x\in B_{L_0}(0).
\end{cases}
\end{equation}
Let $ \tilde u(x,t;q)$ be the solution of \eqref{new-new-eqq1}. Let $ \tilde T_0\geq 1$ be as in Step 2.
We claim that there is $\tilde\epsilon_0>0$ such that for any function $q(x,t)$ which is $C^1$ in $x$ {and  H\"older continuous in $t$} with exponent $0<\theta<\frac{1}{2}$,
\begin{equation}
\label{q-eq1}
\sup_{t\ge 0}\|q(\cdot,t)\|_{C(\bar B_{L_0}(0))}\le
\chi N\tilde M \tilde\epsilon_0
 \end{equation}
 ($\tilde M$ is as in Step 1),  and
 \begin{equation}
 \label{q-eq2}
 \sup_{{t,s\ge 0, t\not =s}}\frac{\|q(\cdot,t)-q(\cdot,s)\|_{{C(\bar B_{L_0}(0))}}}{|t-s|^\theta}\le \chi M M_1
 \end{equation}
  ($M$ and $M_1$ are as in Lemma \ref{persistence-lm}), there holds
\begin{equation}
\label{new-new-eqq2}
 \tilde u(x, \tilde T_0;q)\ge 2 \phi_0(x),\quad x\in B_{L_0}(0).
\end{equation}
In fact, assume this is not true. Then there are $\epsilon_n\to 0$ as $n\to\infty$ and $q_n(x,t)$  satisfying \eqref{q-eq1} and \eqref{q-eq2}    {and $x_n\in B_{L_0}(0)$ such that $ \tilde u(x_n, \tilde T_0;q_{n})< 2 \phi_0(x_n)$.}
%\eqref{new-new-eqq2} does not hold.
Let $u_n(x,t)= \tilde u(x,t;q_n)$. Without loss of generality, we may assume that
$$
u_n(x,t)\to u^*(x,t),\quad \p_{x_j} u_n(x,t)\to \p_{x_j} u^*(x,t)\quad  {\rm as} \quad n\to\infty
$$ locally uniformly in $(x,t)\in \bar B_{L_{0}}(0)\times [0,\infty)$. Note that
$u^*(x,t)= \tilde u(x,t;0)=e^{\lambda_0t}\phi_0(x)$. Hence
$$
u^*(x, \tilde T_0)\ge 3 \phi_0(x),\quad x\in B_{L_0}(0).
$$
This together with the Hopf's Lemma implies that
$$
u_n(x, \tilde T_0)\ge 2 \phi_0(x),\quad x\in B_{L_0}(0),
$$
which is a contradiction. Hence the claim holds true.

Next, without loss of generality, we may assume that
$$
\frac{ 3 a}{4}-b \tilde\epsilon_0\ge \frac{a}{2}.
$$
By Step 1, for
  $0<\epsilon\le \tilde\epsilon_0$,  $t_1+T(\epsilon){+1}\le t< t_2\leq \infty$, and $x\in{B}_{L}(x_0)$,
\begin{align}\label{u-sup-eq}
u_{t}&=\Delta u - \chi \nabla v\cdot\nabla u+u(a-\chi \Delta v -bu)\cr
&\geq \Delta u +q(x,t)\cdot\nabla u+\frac{a}{2}u,
\end{align}
where $q(x,t)=- \chi \nabla v(x,t;u_0,v_0)$. By Step 1 and Lemma \ref{persistence-lm}, $q(\cdot+x_0,\cdot{+t_1+T(\epsilon)+1})$ satisfies \eqref{q-eq1} and \eqref{q-eq2}.
Let $n_0\ge 0$ be such that
  $$t_1+T(\epsilon)+1+ n_0 \tilde T_0<t_2\quad {\rm and}\quad t_1+T(\epsilon)+1+(n_0+1) \tilde T_0\ge t_2.
  $$
 By Step 2,
$$
\inf_{x\in B_{2L}(x_0)} u(x,t;u_0,v_0)\ge \delta_\epsilon\quad \forall\, t_1\le t\le t_1+T(\epsilon)+1.
$$
This together with
 the comparison principle for parabolic equations and \eqref{new-new-eqq2} implies that
\begin{align}
u(x,t_1+T(\epsilon)+1+k \tilde T_0;u_0,v_0)&\geq  2^{k-1}\delta_\epsilon  \tilde u(x-x_0,\tilde T_0;q(\cdot+x_0,\cdot+t_1+T(\epsilon)+1+(k-1) \tilde T_0))\cr
&\geq  2^{k} \delta_{\epsilon} \phi_0(x-x_0) \quad \forall\,\ x\in{B}_{L_0}(x_0)
\end{align}
for $k=1,2,\cdots,n_0$, where $\delta_\epsilon$ is as in Step 2.
By Step 2 again, we then have
$$
\inf_{x\in B_{2L}(x_0)}u(x,t;u_0,v_0)\ge \tilde\delta_\epsilon:=\min\{\delta_\epsilon,\delta_{\delta_\epsilon}\}\quad \forall \, t_1\le t<t_2.
$$
}

\medskip

\noindent {\bf Step 4.}
In this step, we prove that {\it  there is $m>0$ such that for any $(u_0,v_0)\in X_1^+\times X_2^+$
with $\inf_{x\in\R^N}u_0(x)>0$,
$$
\liminf_{t\to\infty} \inf_{x\in\R^N} u(x,t;u_0,v_0)\ge m.
$$
}

 To prove the statement in Step 4, let {$\tilde \epsilon_0$ be as in Step 3 and
$$
\tilde \delta:=\tilde \delta(u_0,v_0)=\inf_{x\in\R^N} u(x,T_0(u_0,v_0)+T(\tilde \epsilon_0)+3;u_0,v_0).
$$}
By the assumption $\inf_{x\in\R^N} u_0(x)>0$, {$\tilde \delta>0$}. Let
$$
k_0=\inf\{k\in \Z\,|\, 2^{k}{\tilde \delta}\ge \tilde\epsilon_0\}
\quad {\rm and}\quad
T_{00}(u_0,v_0)=T_0(u_0,v_0)+T(\tilde\epsilon_0){+3}+k_0  \tilde T_0.
$$
We claim that
\begin{equation}
\label{step5-eq}
\inf_{x\in\R^N}u(x,t;u_0,v_0)\ge \min\{\delta_{\tilde\epsilon_0},\tilde\delta_{\tilde\epsilon_0}\}\quad \forall\, t\ge T_{00}(u_0,v_0).
\end{equation}

To prove the claim,
for any $x_0\in\R^N$, let
$$
I(x_0)=\{t> T_0(u_0,v_0){+2}\,|\, \sup_{x\in{B}_{2L(\tilde\epsilon_0)}(x_0)}u(x,t;u_0,v_0)<\tilde\epsilon_0\}.
$$
Note that $I(x_0)$ is an open set. By Step 2,
\begin{equation}
\label{step5-eq0}
\inf_{x\in B_{2L(\tilde\epsilon_0)}(x_0)}u(x,t;u_0,v_0)\ge \delta_{\tilde\epsilon_0}\quad \forall\, t\not \in I(x_0)\,\,\, {{\rm for}\, t>T_0(u_0,v_0){+2}}.
\end{equation}
Hence, if $I(x_0)=\emptyset$, then
\begin{equation}
\label{step5-eq1}
\inf_{x\in{B}_{2L(\tilde\epsilon_0)}(x_0)} u(x,t;u_0,v_0)\ge \delta_{\tilde\epsilon_0}\quad \forall\,\  t\ge T_0(u_0,v_0){+2}.
\end{equation}

If $I(x_0)\not=\emptyset$, then $I(x_0)=\cup (a_i,b_i)$. If $a_i\not =T_0(u_0,v_0){+2}$,  then
$$
\sup_{x\in{B}_{2L(\tilde\epsilon_0)}(x_0)} u(x, a_i;u_0,v_0)=\tilde\epsilon_0
\quad{\rm and}\quad
\sup_{x\in{B}_{2L(\tilde\epsilon_0)}(x_0)}u(x,t;u_0,v_0)<\tilde\epsilon_0\quad \forall\, t\in (a_i,b_i).
$$
 By the statement in Step 3,
 \begin{equation}
 \label{step5-eq2}
 \inf_{x\in{B}_{2L(\tilde\epsilon_0)}(x_0)} u(x,t;u_0,v_0)\ge \tilde\delta_{\tilde\epsilon_0}\quad \forall\, t\in (a_i,b_i)
 \,\, {\rm for}\,\, a_i\not = T_0(u_0,v_0){+2}.
 \end{equation}
 If $a_i=T_0(u_0,v_0){+2}$,
 by the arguments in Step 3,  there holds
\begin{align*}
u(x,T_0(u_0,v_0)+T(\tilde\epsilon_0)+3+k \tilde T_0;u_0,v_0)\geq  2^{k} \tilde\delta \phi_{0}(x-x_0) \quad \forall\,\ x\in{B}_{L_0}(x_0)
\end{align*}
for $k=0,1,2,\cdots, k_0$. This implies that $b_i\le T_{00}(u_0,v_0)$. This together with \eqref{step5-eq0},
\eqref{step5-eq1}, and \eqref{step5-eq2}
implies \eqref{step5-eq}.
 This proves the statement in Step 4 with $m=\min\{\delta_{\tilde\epsilon_0},\tilde\delta_{\tilde\epsilon_0}\}$.
  Theorem \ref{Main-thm3} is thus proved.
\end{proof}

%%%%%%%%%

\section{Asymptotic behavior of  solutions with strictly positive initial data}
In this section, we discuss the asymptotic behavior of global bounded  classical solutions of \eqref{Main-Eq} and prove Theorem \ref{Main-thm4}. The proof of Theorem \ref{Main-thm4} can be done by following the ideas given in \cite{win_JDE2014}.
Throughout this section, We assume that $\inf_{x\in\R^N}u_{0}(x)>0$ for $u_{0}\in C_{\rm unif}^{b}(\R^N)$ and $v_0(x)\geq 0$ for $v_0 \in C_{\rm unif}^{b,1}(\R^N)$.
We also assume that  $b>\frac{N\chi \mu}{4}$ and $\lambda\geq \frac{a}{2}$, and
$\theta=\frac{N \mu \chi}{4b}$. Hence $0<\theta<1$.   We  denote by $(u(x,t; u_0, v_0), v(x,t; u_0, v_0))$ the global bounded classical solution of \eqref{Main-Eq} associated with initial data $(u_{0}, v_0)$. Again, throughout this section, unless specified otherwise, $C$ denotes a generic  constant
independent of $u,v$ and may be different at different places.

Recall that $U(x,t)=u(x, t; u_0, v_0)-\frac{a}{b}$, $V(x,t)=v(x, t; u_0, v_0)-\frac{\mu}{\lambda}\frac{a}{b}$. Then $(U, V)$ solves \eqref{U-V-eq}.
We first present some lemmas on the estimates about $U$ and $V$. The first lemma provides an estimate on $\|\nabla V(\cdot,t)\|_\infty$.

\begin{lem}\label{v_deriv}
There exists $C_0=C_0(a, \mu, \lambda)>0$ such that
%Let $C_0=C_0(a, \mu, \lambda)=\frac{2a\mu C\sqrt{\pi}}{\sqrt{\lambda}}$. Then
 \begin{equation}
\limsup_{t\to\infty}\|\nabla V(\cdot, t)\|_{\infty}\leq \frac{C_0}{b(1-\theta)}.
\end{equation}
\end{lem}
\begin{proof}
%By \eqref{U-V-estimate} and
By \eqref{u-bdd-2}, we can fix a sufficiently large $t_1$ such that
%\begin{equation}\label{u-U-ineq}
%\max\{\|u(\cdot, t; u_0,v_0)\|_{\infty}, \|U(\cdot,t)\|_\infty\}\leq \frac{2a}{b(1-\theta)}  \quad\,\ \forall\,\ t>t_1.
%\end{equation}
\begin{equation}\label{u-ineq}
\|u(\cdot, t; u_0,v_0)\|_{\infty}\leq \frac{2a}{b(1-\theta)}  \quad\,\ \forall\,\ t>t_1.
\end{equation}
By the variation of constant formula, we have that
$$
v(\cdot, t; u_0, v_0)=e^{(t-t_1)(\Delta-\lambda I)}v(\cdot, t_1; u_0, v_0)+\mu\int_{t_1}^t e^{(t-s)(\Delta-\lambda I)}u(\cdot, s; u_0, v_0)ds \quad \forall\,\ t\geq t_1.
$$
Note that $\nabla V(x,t)=\nabla v(x,t;u_0,v_0)$. Thus, we have
\begin{align}\label{v-deriv}
\|\nabla V(\cdot,t)\|_\infty \leq &\|\nabla e^{(t-t_1)(\Delta-\lambda I)}v(\cdot, t_1;u_0, v_0)\|_{\infty} \nonumber\\
& \, +\mu \int_{t_1}^t \|\nabla e^{(t-s)(\Delta-\lambda I)}u(\cdot, s;u_0,v_0)\|_{\infty}ds \quad \forall\,\ t\geq t_1.
\end{align}
By \eqref{Lp-Lq-2} and \eqref{u-ineq},  we have
\begin{equation}\label{v-deriv1}
\|\nabla e^{(t-t_1)(\Delta-\lambda I)}v(\cdot, t_1;u_0,v_0)\|_{\infty}\leq C(t-t_1)^{-\frac{1}{2}}e^{-\lambda(t-t_1)}\|v(\cdot, t_1;u_0,v_0)\|_{\infty}
\quad \forall\,\ t\geq t_1,
\end{equation}
and
\begin{align}\label{v-deriv2}
\mu \int_{t_1}^t \|\nabla e^{(t-s)(\Delta-\lambda I)}u(\cdot, s;u_0,v_0)\|_{\infty}ds&\leq \mu C\int_{t_1}^t (t-s)^{-\frac{1}{2}}e^{-\lambda(t-s)}\|u(\cdot, s;u_0,v_0)\|_{\infty}ds\cr
&\leq \frac{2a \mu C}{b(1-\theta)} \int_{t_1}^t (t-s)^{-\frac{1}{2}}e^{-\lambda(t-s)}ds\cr
&\leq \frac{2a \mu C\sqrt{\pi}}{b(1-\theta)\sqrt{\lambda}}\quad \forall\, t\ge t_1.
\end{align}
The lemma with $C_0=C_0(a, \mu, \lambda)=\frac{2a\mu C\sqrt{\pi}}{\sqrt{\lambda}}$ then follows from \eqref{v-deriv}, \eqref{v-deriv1} and \eqref{v-deriv2}.
\end{proof}

The second lemma provides an estimate on $\|\Delta V(\cdot,t)\|_\infty$.

\begin{lem}\label{Delta-v-bdd}
There exists $C_1=C_1(\lambda, a, N)>0$ such that the following holds
\begin{equation}\label{delta-v-est}
\limsup_{t\to\infty}\|\Delta V(\cdot,t)\|_{\infty}\leq \frac{\mu C_1}{b(1-\theta)^2}.
\end{equation}
\end{lem}

\begin{proof}
Fix $\beta$ and $\gamma$ such that $\gamma\in(1,\frac{3}{2})$ and $\gamma-1<\beta<\frac{1}{2}$.
We  first prove
that
 there exists $ \tilde C_1= \tilde C_1(\lambda, a, N)>0$ such that
\begin{equation}
\label{A-beta-est}
\limsup_{t\to\infty}\|A^{\beta}U(\cdot, t)\|_{\infty}\leq \frac{ \tilde C_1}{b(1-\theta)^2},
\end{equation}
where  $A^{\beta}=(\lambda I-\Delta)^{\beta}$.
To this end,  let $t_1>0$ be such that \eqref{u-ineq} holds, {then
\begin{align}\label{U-ineq}
\|U(\cdot,t)\|_{\infty}\leq \|U(\cdot,t)-\frac{a}{b}\|_{\infty}+\frac{a}{b}
\leq \frac{2a}{b(1-\theta)}+\frac{a}{b}
\leq \frac{3a}{b(1-\theta)}\quad\,\ \forall\,\ t>t_1.
\end{align}
}

By lemma \ref{v_deriv}, we can fix $t_2\ge t_1$ sufficiently large such that
\begin{equation}\label{v-ineq2}
\|\nabla v(\cdot, t;u_0,v_0)\|_{\infty}=\|\nabla V(\cdot,t)\|_\infty \leq \frac{2C_{0}}{b(1-\theta)} \quad\,\ \forall\,\ t>t_2.
\end{equation}
{By the variation of constant formula,
\begin{align}
U(\cdot, t)&=e^{(t-t_2)(\Delta-a I)}U(\cdot, t_2)-\chi\int_{t_2}^{t} e^{(t-s)(\Delta-a I)}\nabla\cdot(u(\cdot,s;u_0,v_0)\nabla V(\cdot, s))ds\cr
&\,\,\,-b\int_{t_2}^t e^{(t-s)(\Delta-a I)}U(\cdot,s)^{2}ds.
\end{align}}
 By the similar arguments to those in the proof of the last inequality of \eqref{condition-on-initial-eq}, we obtain that
\begin{align*}
\|A^{\beta} U(\cdot, t)\|_{\infty}&\leq C(t-t_2)^{-\beta}e^{-\lambda(t-t_2)}\|U(\cdot, t_2)\|_{\infty}+\chi C a^{\beta+\frac{1}{2}}\frac{4C_{0}}{b^2(1-\theta)^2}+\frac{9C}{b(1-\theta)^2}a^{\beta+1}.
\end{align*}
This together with  $\frac{\chi}{b}=\frac{4\theta}{N\mu}$ and $0<\theta<1$ implies
  \eqref{A-beta-est} with
  $ \tilde C_1= \tilde C_1(\lambda, a, N)=\frac{32Ca^{\beta+\frac{3}{2}}\sqrt{\pi}}{N\sqrt{\lambda}}+9Ca^{\beta+1}$.

Next, by \eqref{A-beta-est}, we can fix $t_3\ge t_2$ sufficiently large such that
\begin{equation}\label{A-beta-bd}
\|A^{\beta}U(\cdot, t)\|_{\infty}\leq \frac{2 \tilde C_1}{b(1-\theta)^2}  \quad\,\ \forall\,\ t>t_3.
\end{equation}
By the variation of constant formula, we have
$$
V(\cdot, t)=e^{(t-t_3)(\Delta-\lambda I)}V(\cdot,t_3)+\mu\int_{t_3}^{t}e^{(t-s)(\Delta-\lambda I)}U(\cdot,s)ds \quad \forall\,\ t>t_3.
$$
Note that
\begin{equation}\label{C2-embeding}
\|\Delta V(\cdot, t)\|_{\infty}\leq C \|A^\gamma V(\cdot, t)\|_{\infty}.
\end{equation}
By \eqref{Lp-Lq-3},
\begin{align}\label{A-gamma-1}
\|A^{\gamma}e^{(t-t_3)(\Delta-\lambda I)}V(\cdot,t_3)\|_{\infty}\leq C_{\gamma}(t-t_3)^{-\gamma}e^{-\lambda(t-t_3)}\|V(\cdot,t_3)\|_{\infty}\to 0 \quad {\rm as} \quad t\to\infty.
\end{align}
By \eqref{Lp-Lq-3} and \eqref{A-beta-bd}, we have
\begin{align}
\mu\int_{t_3}^{t}\|A^{\gamma}e^{(t-s)(\Delta-\lambda I)}U(\cdot,s)\|_{\infty}ds&=\mu\int_{t_3}^{t}\|A^{\gamma-\beta}e^{(t-s)(\Delta-\lambda I)}A^{\beta}U(\cdot,s)\|_{\infty}ds\cr
&\leq \mu C\frac{2 \tilde C_1}{b(1-\theta)^2}\int_{t_3}^{t}(t-s)^{-(\gamma-\beta)}e^{-\lambda(t-s)}ds\cr
&=\mu C\frac{2 \tilde C_1}{b(1-\theta)^2}\lambda^{\gamma-\beta-1}\quad \forall\, t\ge t_3.
\end{align}
It then follows that
\eqref{delta-v-est} holds with $C_1=C_1(\lambda, a, N)=2C \tilde C_1\lambda^{\gamma-\beta-1}$.%=2C\lambda^{\gamma-\beta-1}\big(\frac{32Ca^{\beta+\frac{3}{2}}\sqrt{\pi}}{N\sqrt{\lambda}}+9Ca^{\beta+1}\big)$.
\end{proof}

Observe that, by {\eqref{U-ineq}},
$$
\limsup_{t\to\infty}\|U(\cdot,t)\|_\infty\le \frac{3a}{b(1-\theta)}.
$$
In the following lemma, we provide a better estimate for $\|U(\cdot,t)\|_\infty$.

\begin{lem}\label{u-lower-bdd}
There exists $C_2=C_2(\lambda, a, N)>0$ such that
\begin{equation}\label{u-upper-bound}
\limsup_{t\to\infty}\|U(\cdot,t)\|_{\infty}\leq \frac{ C_2\theta}{b(1-\theta)^2}.
\end{equation}
\end{lem}

\begin{proof}
We first prove that there exists $\tilde C_2=\tilde C_2(\lambda, a, N)>0$ such that
\begin{equation}\label{u-lower-bound}
\liminf_{t\to\infty}\big(\inf_{x\in\R^N}u(x,t;u_0,v_0)\big)\geq \frac{a}{b}-\frac{ \tilde C_2\theta}{b(1-\theta)^2}.
\end{equation}
By Lemma \ref{Delta-v-bdd}, we can choose $t_4$ large enough such that
\begin{equation}
\label{t-4-eq}
\|\Delta v(\cdot,t; u_0, v_0)\|_{\infty}=\|\Delta V(\cdot,t)\|_{\infty}\leq \frac{2\mu C_1}{b(1-\theta)^2} \quad \forall \,\ t\geq t_4.
\end{equation}
Therefore,
\begin{align}
u_{t}\ge \Delta u - \chi \nabla v\cdot\nabla u+u(a-\frac{2\mu C_1\chi}{b(1-\theta)^2})-bu^2 \quad \forall\,\ x\in\R,\,\ t>t_4.
\end{align}
By the comparison principle for parabolic equations, we have
$$
\liminf_{t\to\infty}\big(\inf_{x\in\R^N}u(x,t;u_0,v_0)\big)\geq \frac{a-\frac{2\mu C_1 \chi}{b(1-\theta)^2}}{b}.
$$
This together with  $\frac{\chi}{b}=\frac{4\theta}{N\mu}$ implies that \eqref{u-lower-bound} holds with $\tilde C_2=\tilde C_2(\lambda, a, N)=\frac{8C_1}{N}$.

Next, we prove \eqref{u-upper-bound}.
By \eqref{u-lower-bound},
$$
\limsup_{t\to\infty}\|U_{-}(\cdot, t)\|_{\infty}\leq \frac{ \tilde C_2\theta}{b(1-\theta)^2}.
$$
By \eqref{U-posi-bdd},
$$
\limsup_{t\to\infty}\|U_{+}(\cdot, t)\|_{\infty}\leq \frac{a\theta}{b(1-\theta)}.
$$
Thus, \eqref{u-upper-bound} holds if we let $C_2=C_2(\lambda, a, N)=\max\{\tilde  C_2, a\}$.
\end{proof}

We now prove  Theorem \ref{Main-thm4}.

\begin{proof}[Proof of Theorem \ref{Main-thm4}]
 First of all, let $\theta_{0}\in(0,1)$ be defined
\begin{equation*}
\theta_0=\sup\{ \theta\in (0,1)\,|\, \frac{2C_2\theta}{(1-\theta)^2 a }\le \frac{1}{6}
\quad {\rm and}\quad
\frac{8C \lambda^{-\frac{1}{2}}a^{\frac{1}{2}}\pi \theta}{N(1-\theta)}\le \frac{1}{12}\}.
\end{equation*}
Then
\begin{equation}\label{eq-1-0}
\frac{2C_2\theta_0}{(1-\theta_0)^2 a }\le \frac{1}{6}
\quad {\rm and}\quad
\frac{8C \lambda^{-\frac{1}{2}}a^{\frac{1}{2}}\pi \theta_0}{N(1-\theta_0)}\le \frac{1}{12}.
\end{equation}
Let $K=\frac{N}{4\theta_0}>\frac{N}{4}$.
We prove that for any $b>K\chi\mu$, there are $C>0$ and $\alpha>0$ such that
\begin{equation}
\label{U-exp-decay}
\|U(\cdot,t)\|_{\infty}\leq Ce^{-\alpha t} \quad \forall\,\ t>0.
\end{equation}

To this end, first fix $b>K\chi\mu$. Note that
$\theta=\frac{N\mu\chi}{4b}<\theta_0.$
Then by \eqref{eq-1-0},
 there is $0<\alpha<\min\{\lambda,a\}$ such that
 \begin{equation}\label{eq-1-1}
\frac{2C_2\theta}{(1-\theta)^2 (a-\alpha) }\le  \frac{1}{6}
\quad {\rm and}\quad
\frac{8a C(\lambda-\alpha)^{-\frac{1}{2}}(a-\alpha)^{-\frac{1}{2}}\pi \theta}{N(1-\theta)}\le  \frac{1}{12}.
\end{equation}
Fix  $0<\alpha<\min\{\lambda,a\}$ such that \eqref{eq-1-1} holds and fix  $B>0$ large enough such that
\begin{equation}\label{eq-3-1}
\frac{2C_2\theta}{b(1-\theta)^2} \leq \frac{B}{6}
\quad {\rm and}\quad
\frac{16a C_{0}(a-\alpha)^{-\frac{1}{2}}\sqrt{\pi} \theta}{N\mu b(1-\theta)^2}\leq \frac{B}{12}.
\end{equation}
By \eqref{u-upper-bound}, there exists $t_0\ge t_4(\ge t_3\ge t_2\ge t_1)$  such that
\begin{equation}\label{U-upper}
\|U(\cdot,t)\|_{\infty}\le \frac{2C_2\theta}{b(1-\theta)^2} \quad \forall\,\ t\ge t_0,
\end{equation}
Consider the set
$$
S=\{T_0\ge t_0\, |\, \|U(\cdot,t)\|_{\infty}\leq Be^{-\alpha(t-t_0)}, \,\ \forall\,\ t\in[t_0, T_0]\}.
$$
 By \eqref{eq-3-1} and \eqref{U-upper}, we have $\|U(\cdot,t_0)\|_{\infty}\leq \frac{B}{6}$. Thus, $S$ is not empty and $T:=\sup S\in(t_0,\infty]$ is well-defined. Hence, to prove \eqref{U-exp-decay}, it is sufficient to prove that
\begin{equation}\label{T-infty}
T=\infty.
\end{equation}

Next, by the variation of constant formula,
$$
\|\nabla V(\cdot, t)\|_\infty=\|\nabla e^{(t-t_0)(\Delta-\lambda I)}V(\cdot, t_0)+\mu\int_{t_0}^t \nabla e^{(t-s)(\Delta-\lambda I)}U(\cdot, s)ds\|_\infty \quad \forall\,\ t\geq t_0.
$$
By \eqref{Lp-Lq-1}, \eqref{v-ineq2}, we have
\begin{align}\label{nabla-V-1}
\|\nabla e^{(t-t_0)(\Delta-\lambda I)}V(\cdot, t_0)\|_{\infty}&=\| e^{(t-t_0)(\Delta-\lambda I)} \nabla V(\cdot, t_0)\|_{\infty}\leq e^{-\lambda (t-t_0)}\|\nabla V(\cdot, t_0)\|_{\infty}\cr
&\leq e^{-\lambda (t-t_0)}\frac{2C_{0}}{b(1-\theta)}\leq \frac{2C_{0}}{b(1-\theta)} e^{-\alpha (t-t_0)}\quad \forall\, t\ge t_0.
\end{align}
Furthermore, \eqref {Lp-Lq-2} along with the definition of $T$ gives us that
\begin{align}\label{nabla-V-2}
\mu\int_{t_0}^t \|\nabla e^{(t-s)(\Delta-\lambda I)}U(\cdot, s)\|_{\infty}ds&\leq \mu C\int_{t_0}^t (t-s)^{-\frac{1}{2}}e^{-\lambda(t-s)}\|U(\cdot, s)\|_{\infty}ds\cr
&\le \mu B C\lambda^{-\frac{1}{2}}\big(\int_{0}^{\lambda(t-t_0)}\sigma^{-\frac{1}{2}}e^{-(1-\frac{\alpha}{\lambda})\sigma}d\sigma\big)e^{-\alpha(t-t_0)}\cr
&\leq \mu BC\lambda^{-\frac{1}{2}}(1-\frac{\alpha}{\lambda})^{-\frac{1}{2}}\sqrt{\pi}e^{-\alpha(t-t_0)} \quad\,\ \forall\,\ t\in(t_0, T).
\end{align}
Combing \eqref{nabla-V-1} and \eqref{nabla-V-2}, we get that
\begin{equation}\label{nabla-V}
\|\nabla V(\cdot, t)\|_{\infty}\leq \big\{\frac{2C_{0}}{b(1-\theta)} +\mu BC\lambda^{-\frac{1}{2}}(1-\frac{\alpha}{\lambda})^{-\frac{1}{2}}\sqrt{\pi}\big\}e^{-\alpha(t-t_0)} \quad\,\ \forall\,\ t\in(t_0, T).
\end{equation}

By the variation of constant formula again, we have
\begin{align}\label{U-eq-2}
\|U(\cdot, t)\|_{\infty}&\leq \|e^{(t-t_0)(\Delta-a I)}U(\cdot, t_0)\|_{\infty}+\chi\int_{t_0}^t \|e^{(t-s)(\Delta-a I)}\nabla\cdot(u(\cdot, s;u_0,v_0)\nabla V(\cdot, s))\|_{\infty}ds\cr
&\,\, \,\,\, +b\int_{t_0}^t \|e^{(t-s)(\Delta-a I)}U^2(\cdot, s)\|_{\infty}ds \quad \forall\,\ t>t_0.
\end{align}
It follows from  \eqref{eq-3-1} and  \eqref{U-upper} that
\begin{align}\label{U-1}
\|e^{(t-t_0)(\Delta-a I)}U(\cdot, t_0)\|_{\infty}&\leq e^{-a(t-t_0)}\|U(\cdot, t_0)\|_{\infty}\leq e^{-a(t-t_0)} \frac{2C_2\theta}{b(1-\theta)^2}\cr
&\leq e^{-\alpha(t-t_0)} \frac{2C_2\theta}{b(1-\theta)^2}\leq \frac{B}{6} e^{-\alpha(t-t_0)} \quad \forall\,\ t>t_0.
\end{align}
By Lemma \ref{L_Infty bound 2}, \eqref{u-ineq}, \eqref{eq-1-1}, \eqref{eq-3-1}, \eqref{nabla-V}, and $\frac{\chi}{b}=\frac{4\theta}{N\mu}$, we have
\begin{align}
\label{U-2}
&\chi\int_{t_0}^t \|e^{(t-s)(\Delta-a I)}\nabla\cdot(u(\cdot, s;u_0,v_0)\nabla V(\cdot, s))\|_{\infty}ds\cr
&\leq \chi C\int_{t_0}^t e^{-a(t-s)}(t-s)^{-\frac{1}{2}}\frac{2a}{b(1-\theta)}\big\{\frac{2C_{0}}{b(1-\theta)} +\mu BC\lambda^{-\frac{1}{2}}(1-\frac{\alpha}{\lambda})^{-\frac{1}{2}}\sqrt{\pi}\big\}e^{-\alpha(s-t_0)}ds\cr
&\leq\big\{\frac{16\theta aC_{0}}{bN\mu(1-\theta)^2} +\frac{8\theta a B C(\lambda-\alpha)^{-\frac{1}{2}}\sqrt{\pi}}{N(1-\theta)}\big\}(a-\alpha)^{-\frac{1}{2}}\sqrt{\pi}e^{-\alpha(t-t_0)}\cr
&\leq \frac{B}{6}e^{-\alpha(t-t_0)} \quad\,\ \forall\,\ t\in(t_0, T).
\end{align}
By \eqref{eq-1-1}, \eqref{U-upper},  and the definition of $T$, we have
\begin{align}\label{U-3}
b\int_{t_0}^t \|e^{(t-s)(\Delta-a I)}U^2(\cdot, s))\|_{\infty}ds&\leq b\int_{t_0}^t e^{-a(t-s)}\|U(\cdot, s)\|_{\infty}\|U(\cdot, s)\|_{\infty}ds\cr
&\leq b\int_{t_0}^t e^{-a(t-s)}\frac{2C_2\theta}{b(1-\theta)^2} B e^{-\alpha(s-t_0)}ds\cr
&\leq \frac{2C_2\theta}{(1-\theta)^2}\cdot B\cdot \frac{1}{a-\alpha}\cdot e^{-\alpha(t-t_0)}\cr
&\leq \frac{1}{6} B e^{-\alpha(t-t_0)}\quad \forall\,\ t\in(t_0,T).
\end{align}
Combing \eqref{U-1}, \eqref{U-2} and \eqref{U-3}, we can obtain that
$$
\|U(\cdot, t)\|_{\infty}\leq 3\cdot \frac{1}{6} B e^{-\alpha(t-t_0)}=\frac{B}{2}e^{-\alpha(t-t_0)}\quad \forall\,\ t\in(t_0,T)
$$
which together with the continuity of $U$ implies that $T$ cannot be finite. This shows \eqref{T-infty}, and
\eqref{U-exp-decay} then follows.

We now prove that there is $C>0$ such that
\begin{equation}
\label{V-exp-decay}
\|V(\cdot,t)\|_\infty\le C e^{-\alpha t}\quad \forall\, t>0.
\end{equation}
 By variation of constants formula associated with the second equation in \eqref{U-V-eq}, we get that
$$
V(\cdot, t)=e^{t(\Delta-\lambda I)}\big(v_0-\frac{\mu}{\lambda}\frac{a}{b}\big)+\mu\int_{0}^{t}e^{(t-s)(\Delta-\lambda I)}U(\cdot,s)ds \quad \forall\,\ t>0.
$$
By \eqref{U-exp-decay}, we have
\begin{align}\label{V-decay}
\|V(\cdot, t)\|_{\infty}&\leq \|e^{t(\Delta-\lambda I)}\big(v_0-\frac{\mu}{\lambda}\frac{a}{b}\big)\|_{\infty}+\mu\int_{0}^{t}\|e^{(t-s)(\Delta-\lambda I)}U(\cdot,s)\|_{\infty}ds \cr
&\leq e^{-\lambda t}\|v_0-\frac{\mu}{\lambda}\frac{a}{b}\|_{\infty}+C\mu \int_{0}^{t} e^{-\lambda (t-s)}e^{-\alpha s}ds\cr
&= e^{-\lambda t}\|v_0-\frac{\mu}{\lambda}\frac{a}{b}\|_{\infty}+\frac{C\mu}{\lambda-\alpha}(e^{-\alpha t}-e^{-\lambda t}) \quad \forall\,\ t>0.
\end{align}
\eqref{V-exp-decay} then follows, and  \eqref{U-exp-decay} and \eqref{V-exp-decay} establish \eqref{u-v-decay}.
\end{proof}


\begin{thebibliography}{9}




\bibitem{BBTW}
  N. Bellomo, A. Bellouquid,  Y. Tao, and M. Winkler,
   Toward a mathematical theory of Keller-Segel models of pattern formation in biological tissues,
  {\it  Math.~Models Methods Appl.~Sci.}, {\bf 25} (2015), 1663-1763.





\bibitem{CsPj} S. Childress, J. K. Percus, Nonlinear aspects of chemotaxis, {\it Math. Biosci.}, {\bf 56}(1981) 217-237.




\bibitem{DiNa} J.I. Diaz and T. Nagai, Symmetrization in a parabolic-elliptic system related to chemotaxis,
{\it Advances in Mathematical Sciences and Applications}, {\bf 5} (1995), 659-680.


\bibitem{DiNaRa} J.I. Diaz, T.  Nagai, J.-M. Rakotoson, Symmetrization Techniques on Unbounded Domains: Application to a Chemotaxis System on $\R^{N}$,
 {\it J. Differential Equations}, {\bf 145} (1998),  156-183.





\bibitem{GaSaTe}  E. Galakhov, O. Salieva and J. I. Tello,  On a Parabolic-Elliptic system with Chemotaxis and logistic type growth, {\it J. Differential Equations}, {\bf 261} (2016) 4631-4647.


\bibitem{Dan Henry} D. Henry, Geometric Theory of Semilinear Parabolic Equations, Springer-Verlag Berlin Heidelberg New York, 1981.


\bibitem{HiPa} T. Hillen and  K.J. Painter, A User’s Guide to PDE Models for Chemotaxis, {\it J. Math. Biol.} {\bf 58} (2009) (1), 183-217.



\bibitem{Hor} D. Horstmann,
From 1970 until present: the Keller-Segel model in chemotaxis and its consequences,
{\it Jahresber. Dtsch. Math.-Ver.}, {\bf 105} (2003),  103-165.


\bibitem{HorWan} D. Horstmann, G. Wang, Blow-up in a chemotaxis model without symmetry assumptions, {\it European J. Appl. Math.} {\bf 12} (2001), 159–177.




\bibitem{IsSh} T. B. Issa and W. Shen, Pointwise persistence in full chemotaxis models with logistic
source on bounded heterogeneous environments, {\it J. Math. Anal. Appl.}, {\bf 490}, (2020) 124204 .


\bibitem{KeSe1} E.F. Keller and L.A. Segel, Initiation of slime mold aggregation viewed as an instability, {\it J. Theoret. Biol.},
{\bf 26} (1970), 399-415.

\bibitem{KeSe2} E.F. Keller and L.A. Segel,  A  Model for chemotaxis, {\it  J. Theoret. Biol.}, {\bf 30} (1971),  225-234.




\bibitem{lankeit_exceed}
 J. Lankeit,  Chemotaxis can prevent thresholds on population density,
  {\it  Discr.~Cont.~Dyn.~Syst. B},  {\bf 20} (2015), 1499-1527.

\bibitem{lankeit_eventual}
   J. Lankeit,  Eventual smoothness and asymptotics in a three-dimensional chemotaxis system with logistic source,
  {\it  J.~Differential Eq.}, {\bf 258} (2015), 1158-1191.




\bibitem{LiMu} K. Lin and C. L. Mu, Global dynamics in a fully parabolic chemotaxis system with logistic source, {\it Discrete Contin. Dyn. Syst.}, {\bf 36} (2016), 5025-5046.




\bibitem{NAGAI_SENBA_YOSHIDA} T. Nagai, T. Senba and K, Yoshida, Application of the Trudinger-Moser Inequality  to a Parabolic System of Chemotaxis, {\it Funkcialaj Ekvacioj}, {\bf 40} (1997), 411-433.



\bibitem{NtSrUm} T. Nagai, R. Syukuinn and M. Umesako, Decay properties and asymptotic profiles of bounded solutions to a parabolic system of chemotaxis in $\R^N$. {\it Funkcialaj Ekvacioj}, {\bf 46} (2003), 383-407.

\bibitem{NtYt} T. Nagai and T. Yamada, Large time behavior of bounded solutions to a parabolic system of chemotaxis in the whole space, {\it J. Math. Anal. Appl.}, {\bf 336} (2007), 704-726.




\bibitem{OTYM2002}
  K. Osaki, T. Tsujikawa, A. Yagi, and M. Mimura,   Exponential attractor for a chemotaxis-growth system of equations,
   {\it Nonlinear Analysis},  {\bf 51} (2002), 119-144.


\bibitem{OsYa} K. Osaki, A. Yagi, Finite dimensional attractors for one-dimensional Keller–Segel equations, {\it Funkcial. Ekvac.}, {\bf 44} (2001), 441–469.



\bibitem{KJPainter} K.J. Painter, Mathematical models for chemotaxis1 and their applications in self organisation
phenomena, {\it Journal of Theoretical Biology}, {\bf 481} (2019), 162-182.

\bibitem{PaHi} K.J. Painter and  T. Hillen,
Spatio-temporal chaos in a chemotaxis model
{\it Phys. D}, {\bf 240} (2011),  363-375.


\bibitem{A. Pazy} A. Pazy, Semigroups of Linear Operators and Applications to Partial Differential Equations,  Applied Mathematical Sciences, 44. Springer-Verlag, New York, 1983.



\bibitem{SaSh1} R. B. Salako and Wenxian Shen, Global existence and asymptotic behavior of classical solutions to a parabolic-elliptic chemotaxis system with logistic source on $\mathbb{R}^N$, {\it J. Differential Equations}, {\bf 262} (2017) 5635-5690.


\bibitem{SaSh2}  R. B. Salako and Wenxian Shen, Spreading Speeds and Traveling waves of a parabolic-elliptic chemotaxis system with logistic source on $\mathbb{R}^N$, {\it Discrete and Continuous Dynamical Systems - Series A},  {\bf 37} (2017), pp. {6189-6225}.


\bibitem{SaSh3} R. B. Salako and Wenxian Shen, Parabolic-elliptic chemotaxis model with space–time-dependent logistic sources on $\mathbb{R}^N$. I. Persistence and asymptotic spreading, {\it Mathematical Models and Methods in Applied Sciences}, Vol. 28, No. 11, (2018), pp. 2237-2273.





 \bibitem{SaShXu}  R. B. Salako, W. Shen, and S. Xue, Can chemotaxis speed up or slow down
 the spatial spreading in parabolic-elliptic Keller-Segel systems with logistic source?, {\it J. Math. Biol.}, {\bf 79} (2019), {1455–1490}







\bibitem{TaWi}
   Y. Tao and M. Winkler,  Persistence of mass in a chemotaxis system with logistic source,
   {\it  J. Differential Eq.},  {\bf 259} (2015), 6142-6161.


\bibitem{TeWi} J. I. Tello and M. Winkler, A Chemotaxis System with Logistic Source,
 {\it Communications in Partial Differential Equations}, {\bf 32} (2007), 849-877.






\bibitem{WaMuZh} L. Wang, C. Mu, and P. Zheng, On a quasilinear parabolic-elliptic chemotaxis system with logistic source,
{\it J. Differential Equations}, {\bf 256} (2014), 1847-1872.




\bibitem{Win} M. Winkler, Chemotaxis with logistic source: Very weak global solutions and their boundedness properties, {\it J. Math. Anal. Appl.}
{\bf 348} (2008), 708-729.

\bibitem{win_CPDE2010}
   M. Winkler, Boundedness in the higher-dimensional parabolic-parabolic
  chemotaxis system with logistic source,
  {\it Comm.~Part.~Differential Eq.},  {\bf 35} (2010), 1516-1537.

\bibitem{win_jde}
   M. Winkler,  Aggregation vs.~global diffusive behavior in the higher-dimensional Keller-Segel model,
    {\it Journal of Differential Equations},  {\bf 248} (2010), 2889-2905.

\bibitem{win_JMAA_veryweak}
   M. Winkler,  Blow-up in a higher-dimensional chemotaxis system despite logistic growth restriction,
    {\it Journal of Mathematical Analysis and Applications}, {\bf 384} (2011), 261-272.
\bibitem{win_JMPA}
   M. Winkler, Finite-time blow-up in the higher-dimensional parabolic-parabolic Keller-Segel system,
    {\it J.~Math.~Pures Appl.},  {\bf 100} (2013), 748-767.
\bibitem{win_JDE2014}
   M. Winkler,  Global asymptotic stability of constant equilibria in a fully parabolic chemotaxis system
  with strong logistic dampening,
  {\it J. Differential Eq.}, {\bf 257} (2014), 1056-1077.
\bibitem{win_JNLS}
   M. Winkler,  How far can chemotactic cross-diffusion enforce exceeding carrying capacities?
  {\it J. Nonlinear Sci.}, {\bf 24} (2014), 809-855.



\bibitem{YoYo} T. Yokota and N. Yoshino, Existence of solutions to chemotaxis dynamics with logistic source,
  Discrete Contin. Dyn. Syst. 2015, Dynamical systems, differential equations and applications. 10th AIMS Conference. Suppl., 1125-1133.



\bibitem{ZhLiBaZo} J. Zheng, Y. Y. Li, G. Bao, and X. Zou, A new result for global existence and boundedness
of solutions to a parabolic-parabolic Keller-Segel system with logistic source, {\it Journal of
Mathematical Analysis and Applications}, {\bf 462} (2018), 1-25.




\bibitem{ZhMuHuTi} P. Zheng, C. Mu, X. Hu, and Y. Tian, Boundedness of solutions in a chemotaxis system with nonlinear sensitivity and logistic source,
{\it J. Math. Anal. Appl.}, {\bf 424} (2015), 509-522.


\end{thebibliography}
\end{document}